\documentclass[11pt]{amsart}
\usepackage{graphicx}
\usepackage{amssymb,amsmath,amsthm}
\usepackage{epstopdf}
\usepackage{color}
\usepackage{mdframed}
\usepackage{empheq}
\usepackage{bm}
\usepackage{hyperref}
\usepackage{mathabx} %
\usepackage{enumitem}
\usepackage{xcolor}
\usepackage[linesnumbered,ruled,vlined]{algorithm2e}
\usepackage{subcaption}
\usepackage{bbm}
    
\usepackage{booktabs}
\renewcommand{\arraystretch}{1.2} 
\newcommand{\ra}[1]{\renewcommand{\arraystretch}{#1}}

\newtheorem{theorem}{Theorem}
\newtheorem{definition}[theorem]{Definition}

\newtheorem{proposition}[theorem]{Proposition}
\newtheorem{remark}[theorem]{Remark}
\newtheorem{corollary}[theorem]{Corollary}

\newcommand\bN{\mathbf N}

\newcommand\bW{\mathbf W}
\newcommand\bX{\mathbf X}

\newcommand\br{\mathbf r}
\newcommand\bu{\mathbf u}

\newcommand\bw{\mathbf w}

\newcommand\bz{\mathbf z}

\def\balpha{\boldsymbol{\alpha}}

\newcommand\cC{\mathcal C}
\newcommand\cD{\mathcal D}

\newcommand\cF{\mathcal F}

\newcommand\cL{\mathcal L}
\newcommand\cM{\mathcal M}

\newcommand\cP{\mathcal P}

\newcommand\EE{\mathbb E}
\newcommand\FF{\mathbb F}
\newcommand\JJ{\mathbb J}
\newcommand\PP{\mathbb P}

\newcommand\RR{\mathbb R}
\newcommand\TT{\mathbb T}
\newcommand\ZZ{\mathbb Z}

\makeatletter
\renewcommand\l@paragraph[2]{}
\renewcommand\l@subparagraph[2]{}
\makeatother
\title[Machine Learning for Mean Field Ergodic Control]{Convergence Analysis of Machine Learning Algorithms for the Numerical Solution of Mean Field Control and Games: I the Ergodic Case.
}
\author{Ren\'e Carmona \& Mathieu Lauri\`ere}
\address{Program in Applied and Computational Mathematics \& ORFE}

\begin{document}

\maketitle

\begin{abstract}

We propose two algorithms for the solution of the optimal control of ergodic McKean-Vlasov dynamics. Both algorithms are based on approximations of the theoretical solutions by neural networks, the latter being characterized by their architecture and a set of  parameters. This allows the use of modern machine learning tools, and efficient implementations of stochastic gradient descent.

The first algorithm is based on the idiosyncrasies of  the ergodic optimal control problem.
We provide a mathematical proof of the convergence of the approximation scheme, and we analyze rigorously the approximation by controlling the different sources of error. The second method is an adaptation of the deep Galerkin method to the system of partial differential equations issued from the optimality condition. 

We demonstrate the efficiency of these algorithms on several numerical examples, some of them being chosen to show that our algorithms succeed where existing ones failed. We also argue that both methods can easily be applied to problems in dimensions larger than what can be found in the existing literature. 
Finally, we illustrate the fact that, although the first algorithm is specifically designed for mean field control problems, the second one is more general and can also be applied to the partial differential equation systems arising in the theory of mean field games.
\end{abstract}



\section{Introduction}
\label{se:introduction}

The purpose of this paper is to develop numerical schemes for the solution of Mean Field Games (MFGs) and Mean Field Control (MFC) problems. 
The mathematical theory of these problems has attracted a lot of attention in the last decade (see e.g.~\cite{MR2295621,Cardaliaguet-2013-notes,MR3134900,CarmonaDelarue_book_I,CarmonaDelarue_book_II}), and from the numerical standpoint several methods have been proposed, see e.g.~\cite{MR2679575,MR2888257,MR3148086,BricenoAriasetalCEMRACS2017,MR3914553} and~\cite{MR3501391,MR3575615,MR3619691,balata2019class} for finite time horizon MFG and MFC respectively, and~\cite{MR3698446,MR3882530,MR3772008} for stationary MFG. 
However, despite recent progress, the numerical analysis of these problems is still lagging behind because of their complexity, in particular when the dimension is high. 
Here, we choose a periodic model to demonstrate that powerful tools developed for machine learning applications can be harnessed to produce efficient numerical schemes performing better than existing technology in the solution of these problems. We derive systematically the mathematical formulation of the optimization problem amenable to the numerical analysis, and we prove rigorously the convergence of a numerical scheme based on feed-forward neural network architectures. 
Our first method is designed for the optimal control of McKean-Vlasov dynamics, which is the primary purpose of the present work. Besides the intrinsic motivations for this type of problems, a large class of MFGs has a variational structure and can be recast in this form, see e.g.~\cite{MR2295621,MR3644590}. Furthermore, the second method we present tackles the PDE system characterizing optimality conditions satisfied by the solution, and it can be directly adapted to solve the PDE system arising in MFGs as we shall explain.

In the subsequent analysis of finite horizon mean field control problems, see \cite{CarmonaLauriere_DL}, the thrust of the study will be the numerical solution of Forward-Backward Stochastic Differential Equations (FBSDEs) of the McKean-Vlasov type, which generalizes to the mean-field setting techniques introduced in~\cite{MR3736669} for standard BSDEs.  Indeed, the well established probabilistic approach to MFGs and MFC posits that the search for Nash equilibria for MFGs, as well as the search for optimal controls for MFC problems, can be reduced to the solutions of FBSDEs of this type. See for example the books \cite{CarmonaDelarue_book_I,CarmonaDelarue_book_II} for a comprehensive expos\'e of this approach. 
Here, we concentrate on the ergodic problem for which we can provide a direct analytic approach. 
Our mathematical analysis of the model leads to an infinite dimensional optimization problem for which we can
identify and implement numerical schemes capable of providing stable numerical solutions. We prove the theoretical convergence of these approximation schemes and we demonstrate the efficiency of their implementations by comparing their outputs to solutions of benchmark models obtained either by analytical formulas or by a deterministic method for Partial Differential Equations (PDEs). Note that when a game is potential, the search for Nash equilibria can be reformulated as an optimal control problem for a central planner, and because of the McKean-Vlasov nature of this control problem, it can be reformulated as a deterministic control problem for which the state is a probability distribution. This control problem can be tackled by dynamic programing, leading to a Hamilton-Jacobi equation.  This is the approach taken in \cite{Chow_et_al}. See also \cite{Ruthotto_et_al} for a recent contribution.

The reasons for our choice of the ergodic case as a prime testbed for our numerical schemes are twofold. First, the absence of the time variable lowers the complexity of the problem and gives us the opportunity to avoid the use of FBSDEs and to rely on strong approximation results from the theory of feed-forward neural networks which we use in this paper. Second, the objective function can be expressed as an integral over the state space with respect to the invariant measure of the controlled system, leading to a much simpler deterministic optimization problem. Indeed, after proving that the state dynamics at the optimum are given by a gradient diffusion, we postulate the form of the invariant measure and optimize accordingly. Last, the choice of the ergodic case for the first model which we consider is motivated by a forthcoming work on reinforcement learning~\cite{CarmonaLauriereTan-2019-LQMFRL}.

As a final remark we emphasize that, while all the results and numerical implementations concern Markovian controls and equilibria, in most cases, both theoretical and numerical results still hold for controls and strategies being feedback functions of the history of the path of the state of the system. The theoretical extensions are straightforward, and the numerical implementations rely on the so-called \emph{recurrent neural networks} instead of the standard feed-forward networks. We refrain from discussing these extensions to avoid the extra technicalities, especially in the notations and the statements of the results.

The paper is structured as follows. In Section~\ref{se:ergodic}, we present the framework of ergodic mean field control, derive formally necessary optimality conditions and introduce a setting amenable to numerical computations. To wit, we provide the mathematics that lead to natural introductions of the algorithms and their analyses. The first algorithm is presented in Section~\ref{sub:periodic_numerics}. We study rigorously its convergence and its accuracy by proving bounds on the approximation and estimation errors (see Theorems~\ref{thm:ergo-approx-NN} and~\ref{thm:ergo-estim-NN} respectively). Section~\ref{sec:DGM} is dedicated to the second of our algorithms. It is based on a variation over the deep Galerkin solver for the PDE system stemming from the aforementioned optimality conditions. Computational results are presented in Section~\ref{sec:num-res}. They demonstrate the applicability and the performance of our algorithms.
Several test cases are considered. They were chosen for comparison purposes because they can be solved either \emph{explicitly} via analytical formulas, or numerically by classical PDE system solvers like in the case of ergodic mean field games.

\vskip 2pt\noindent
\emph{Acknowledgements:  Both authors were partially supported by ARO grant AWD1005491 and NSF award
AWD1005433. Also, we would like to thank two anonymous referees for a rigorous review of the original version of the paper. Their insightful comments helped us improve significantly the quality of the paper.}

\section{Ergodic Mean Field Control}
\label{se:ergodic}
Since we are not aiming at the greatest possible generality, for the sake of definiteness, we work with a standard infinite horizon drift-controlled It\^o process:
$$
dX_t=b(X_t,\cL(X_t),\alpha_t)dt + dW_t
$$
where $\bW=(W_t)_{t\ge 0}$ is a $d$-dimensional Wiener process and where we use the notation $\cL(X_t)$ for the law of the random variable $X_t$. For notational convenience only, we take the volatility coefficient to be $1$. We limit ourselves to stationary controls $\balpha=(\alpha_t)_{t\ge 0}$ of the form $\alpha_t=\phi(X_t)$ given by deterministic measurable and time-independent feedback functions $\phi$ taking values in a closed convex subset $A$ of a Euclidean space $\RR^k$.  We shall assume that the function $b$ is measurable and bounded on $\RR^d\times \cP_2(\RR^d)\times A$. The space $\cP_2(\RR^d)$ is the space of probability measures on $\RR^d$ having a finite second moment. We assume that it is equipped with the $2$-Wasserstein distance and the corresponding Borel $\sigma$-field. Since we consider controls in feedback form, the above controlled state evolution is in fact a stochastic differential equation, but according to Veretennikov's classical result, this equation:
\begin{equation}
\label{fo:dynamics}
dX_t=b(X_t,\cL(X_t),\phi(X_t))dt +  dW_t
\end{equation}
has a unique strong solution. See for example \cite{Veretennikov, Veretennikov2} and \cite[Theorem 2]{MishuraVeretennikov}. We say that the feedback control function $\phi$ is admissible if it is continuous and if the solution $\bX=(X_t)_{t\ge 0}$ is ergodic in the sense that it has a unique invariant measure which we denote by $\nu^\phi$, and that $\cL(X_t)$ converges toward $\nu^\phi$ in $\cP_2(\RR^d)$. The ergodic theory of McKean-Vlasov stochastic differential equations has recently received a lot of attention. See for example \cite{Butkovsky,Veretennikov2,Yarykin} for some specific ergodicity sufficient conditions.

\subsection{Ergodic Mean Field Costs}
The goal of the ergodic control problem is to minimize the cost:
$$
	J(\phi)=\limsup_{T\to\infty}\frac1T\EE\Bigl[\int_0^T f\bigl(X_t,\cL(X_t),\phi(X_t)\bigr)dt\Bigr].
$$
For the sake of definiteness, we assume that the running cost function $f:\RR^d\times\cP_2(\RR^d)\times A \ni (x, \mu, \alpha) \mapsto f(x, \mu, \alpha) \in \RR$ is continuous and bounded.
The cost can be rewritten in the form:
\begin{equation}
\label{fo:ergodic_cost}
J(\phi)=\limsup_{T\to\infty}\frac1T\int_0^T \langle f\bigl(\cdot,\mu_t,\phi(\cdot)\bigr),\mu_t\rangle  dt,\qquad \mu_t=\cL(X_t),
\end{equation}
if we use the standard notation $\langle\varphi, \nu\rangle $ for the integral of the function $\varphi$ with respect to the measure $\nu$.
When $\phi$ is admissible, the invariant measure $\nu^\phi$ appears as the limit as $t\to\infty$ of $\mu_t$, and if
$f$ is  uniformly continuous in the measure argument, uniformly with respect to the other two arguments, then we can take the limit $T\to\infty$ in the formula giving the ergodic cost \eqref{fo:ergodic_cost} and obtain:
\begin{equation}
\label{fo:ergodic_cost_2}
J(\phi)= \langle f\bigl(\cdot,\nu^\phi,\phi(\cdot)\bigr),\nu^\phi\rangle =F(\nu^\phi,\phi)
\end{equation}
if, for each probability measure $\mu\in\cP_2(\RR^d)$,  and each time-independent $A$-valued feedback function $\phi$ on $\RR^d$, we  use the notation: 
\begin{equation}
\label{fo:F}
F(\mu,\phi)=\int f\bigl(x,\mu,\phi(x)\bigr)\;\mu(dx).
\end{equation}
The controlled process solving \eqref{fo:dynamics} being ergodic, we can
characterize the unique invariant probability measure as the solutions of the non-linear Poisson equation:
\begin{equation}
\label{fo:Poisson}
\frac12 \Delta \nu -\mathrm{div} \bigl( b(\cdot,\nu,\phi(\cdot))\nu \bigr)=0.
\end{equation}
So the goal of our mean field control problem is to minimize, over the admissible  feedback functions $\phi$, the quantity:
\begin{equation}
\label{fo:Jofphi}
J(\phi)= F(\nu^\phi,\phi),
\end{equation}
with $F$ defined above in \eqref{fo:F}, and $\nu^\phi$ solving the Poisson equation \eqref{fo:Poisson}.

\subsection{The Adjoint Equation and Optimality Conditions}
In order to characterize the minima of the functional $J$, we compute its Gateaux derivative. 
To do so, we assume that the function $b:\RR^d\times\cP_2(\RR^d)\times A\ni (x,\mu,\alpha)\mapsto b(x,\mu,\alpha)\in\RR^d$ is continuously differentiable in the variables $(x,\alpha)\in\RR^d\times A\subset \RR^d\times\RR^k$ and has a continuous functional (i.e. linear) differential in the variable $\mu$ when $\mu$ is viewed as an element of the (linear) space $\cM(\RR^d)$ of finite signed measures on $\RR^d$. We denote this derivative by $D_\mu b$. We stress that it is different from the Wasserstein derivative or L-derivative in the sense of Lions. 

Let $\phi$ be fixed and let $\psi$ provide a small perturbation of $\phi$. We first compute, at least formally, the derivative of the probability $\nu^\phi$ in the direction $\psi$, namely:
\begin{equation}
\label{fo:delta_nu}
\delta\nu^{\phi,\psi}=\lim_{\epsilon\searrow 0}\frac1\epsilon [\nu^{\phi+\epsilon\psi}-\nu^\phi],
\end{equation}
when we view probability measures as elements of the space $\cM(\RR^d)$ of finite (signed) measures on $\RR^d$.
Notice that since $\int\nu^\phi=\int \nu^{\phi+\epsilon\psi}=1$, we must have $\int\delta\nu^{\phi,\psi}=0$.  The Poisson equation \eqref{fo:Poisson} implies:
\begin{equation*}
\begin{split}
0&=\frac12\Delta[\nu^{\phi+\epsilon\psi}-\nu^\phi] 
-\mathrm{div}\bigl[\, b(\cdot,\nu^{\phi+\epsilon\psi},\phi(\cdot)+\epsilon\psi(\cdot))\, [\nu^{\phi+\epsilon\psi}-\nu^\phi]\bigr] \\
&\hskip 45pt
-\mathrm{div}\bigl[\, [b(\cdot,\nu^{\phi+\epsilon\psi},\phi(\cdot)+\epsilon\psi(\cdot))-b(\cdot,\nu^{\phi+\epsilon\psi},\phi(\cdot))]\,\nu^\phi]\bigr] \\
&\hskip 75pt
-\mathrm{div}\bigl[\, [b(\cdot,\nu^{\phi+\epsilon\psi},\phi(\cdot))-b(\cdot,\nu^{\phi}\phi(\cdot))]\,\nu^\phi]\bigr]
\end{split}
\end{equation*}
and from this equality, we find that if the directional derivative \eqref{fo:delta_nu} exists, it must solve the following Partial Differential Equation (PDE):
\begin{equation}
\label{fo:PDE}
0=\frac12\Delta(\delta\nu^{\phi,\psi})-\mathrm{div}\bigl[\, b(\cdot,\nu^\phi,\phi(\cdot))\, ( \delta\nu^{\phi,\psi})\bigr] 
-\mathrm{div}\bigl[\partial_\alpha b(\cdot,\nu^\phi,\phi(\cdot))\psi(\cdot)\nu^\phi\bigr]
-\mathrm{div}\bigl[\langle D_\mu b(\cdot,\nu^\phi,\phi(\cdot))\psi(\cdot),\delta\nu^{\phi,\psi}\rangle \nu^\phi\bigr],
\end{equation}
just by dividing both sides by $\epsilon$ and taking the limit $\epsilon\searrow 0$. Note that the quantity 
$$
\langle \partial_\mu b(\cdot,\nu^\phi,\phi(\cdot))\psi(\cdot),\delta\nu^{\phi,\psi}\rangle 
$$ 
is merely the integral of the function $\partial_\mu b(\cdot,\nu^\phi,\phi(\cdot))\psi(\cdot)$ with respect to the measure $\delta\nu^{\phi,\psi}$.

\vskip 6pt
Before we turn to the objective function $J$, we introduce the notion of adjoint equation and adjoint function.

\begin{definition}
\label{de:adjoint}
For each admissible feedback function $\phi$ (and associated solution $\nu^\phi$ of the Poisson equation), we say that the couple $(p,\lambda)$ where $p$ is a function on the state space and $\lambda$ is a constant, is a couple of adjoint variables if they satisfy the following 
linear elliptic PDE:
\begin{equation}
\label{fo:adjoint}
\lambda + \frac12\Delta p(x) + b(x,\nu^\phi,\phi(x))\cdot\nabla p(x)=f(x,\nu^\phi,\phi(x)) + \int D_\mu f(\xi,\nu^\phi,\phi(\xi))(x)\, \nu^\phi(d\xi),
\end{equation}
which we call the adjoint equation. Any solution will be denoted by $(p^\phi,\lambda^\phi)$.
\end{definition}

Recall that here, the derivative $D_\mu f$ is the standard linear functional derivative (of smooth functions on the vector space $\cM(\RR^d)$), which is a function of $x$.%

\begin{proposition}
\label{pr:Gateaux}
The directional derivative of the cost function $J$ defined in \eqref{fo:Jofphi} is given by the formula:
\begin{equation}
\label{fo:dJ_H}
\frac{d}{d\epsilon}J(\phi+\epsilon\psi)\Big|_{\epsilon=0}
=\delta^{\phi,\psi}H(\nu^\phi,p^\phi,\phi),
\end{equation}
where $\delta^{\phi,\psi}$ denotes the functional derivative with respect to $\phi$ in the direction $\psi$, and the Hamiltonian $H$ is defined by:
\begin{equation}
\label{fo:Hamiltonian}
H(\mu,p,\phi)
= F(\mu,\phi)+\int  p(x) \mathrm{div}[b(\cdot,\mu,\phi(\cdot))\mu](x) dx
\end{equation}
\end{proposition}

\begin{proof}
Using the definitions \eqref{fo:Jofphi} and \eqref{fo:F} we get:
\begin{equation*}
\begin{split}
J(\phi+\epsilon\psi)-J(\phi)
&=\int f(x, \nu^{\phi+\epsilon\psi},\phi(x)+\epsilon\psi(x))\nu^{\phi+\epsilon\psi}(dx)-\int f(x, \nu^{\phi},\phi(x))\nu^{\phi}(dx)\\
&=\int\bigl[ f(x, \nu^{\phi+\epsilon\psi},\phi(x)+\epsilon\psi(x))- f(x, \nu^{\phi},\phi(x)+\epsilon\psi(x))\bigr]\nu^{\phi+\epsilon\psi}(dx)\\
&\hskip 45pt
+\int\bigl[ f(x, \nu^{\phi},\phi(x)+\epsilon\psi(x))- f(x, \nu^{\phi},\phi(x))\bigr]\nu^{\phi+\epsilon\psi}(dx)\\&\hskip 45pt
+\int f(x, \nu^{\phi},\phi(x))\bigl[ \nu^{\phi+\epsilon\psi}-\nu^{\phi}\bigr](dx),
\end{split}
\end{equation*}
so that, using Fubini's theorem we have:
\begin{equation*}
\begin{split}
\frac{d}{d\epsilon}J(\phi+\epsilon\psi)\Big|_{\epsilon=0}
&=\int\int D_\mu f(x, \nu^{\phi},\phi(x))(\xi)\;(\delta\nu^{\phi,\psi})(d\xi)\,\nu^\phi(dx)\\
&\hskip 45pt
+\int \partial_\alpha f(x, \nu^{\phi},\phi(x))\psi(x)\,\nu^{\phi}(dx)
+\int f(x, \nu^{\phi},\phi(x))\; (\delta \nu^{\phi,\psi})(dx)\\
&=\int\Bigl[ \int D_\mu f(y, \nu^{\phi},\phi(y))(x)\nu^\phi(dy)+f(x, \nu^{\phi},\phi(x))\Bigr]\; (\delta \nu^{\phi,\psi})(dx)\\
&\hskip 45pt
+\int \partial_\alpha f(x, \nu^{\phi},\phi(x))\psi(x)\,\nu^{\phi}(dx).
\end{split}
\end{equation*}
Now, using the adjoint equation \eqref{fo:adjoint} and the fact that the integral of $\delta\nu^{\phi,\psi}$ is $0$, we get:
\begin{equation*}
\begin{split}
\frac{d}{d\epsilon}J(\phi+\epsilon\psi)\Big|_{\epsilon=0}
&=\int\Bigl[\lambda+ \frac12\Delta p(x) + b(x,\nu^\phi,\phi(x))\cdot\nabla p(x)\Bigr]\; (\delta \nu^{\phi,\psi})(dx)\\
&\hskip 45pt
+\int \partial_\alpha f(x, \nu^{\phi},\phi(x))\psi(x)\,\nu^{\phi}(dx)\\
&=\int p(x)  \frac12\Delta (\delta \nu^{\phi,\psi})(x) dx+\int b(x,\nu^\phi,\phi(x))\cdot\nabla p(x)\; (\delta \nu^{\phi,\psi})(dx)\\
&\hskip 45pt
+\int \partial_\alpha f(x, \nu^{\phi},\phi(x))\psi(x)\,\nu^{\phi}(dx).
\end{split}
\end{equation*}
Finally, using \eqref{fo:PDE} we get:
\begin{equation}
\label{fo:dJ}
\begin{split}
	\frac{d}{d\epsilon}J(\phi+\epsilon\psi)\Big|_{\epsilon=0}
	&=\int  p(x) \Bigl(\mathrm{div}[ \partial_\alpha b(\cdot,\phi(\cdot))\psi(\cdot) \nu^{\phi}](x)
	+\mathrm{div}[ \langle D_\mu b(\cdot,\nu^\phi,\phi(\cdot))\psi(\cdot),\delta\nu^{\phi,\psi}\rangle  \nu^{\phi}](x)\Bigr) dx\\
	&\hskip 75pt
	+\int \partial_\alpha f(x, \nu^{\phi},\phi(x))\psi(x)\,\nu^{\phi}(dx).
\end{split}
\end{equation}
To complete the proof, we express this directional derivative in terms of the Hamiltonian function defined in \eqref{fo:Hamiltonian}.
The latter can be rewritten as:
$$
H(\mu,p,\phi)
= F(\mu,\phi) -\int \nabla p(x) b(x,\mu,\phi(x))\mu(dx),
$$
and its directional derivative is given by:
\begin{equation}
\label{fo:delta_Hamiltonian}
\begin{split}
	\delta^{\phi,\psi}H(\mu,p,\phi)
	&= \lim_{\epsilon\searrow 0}\frac1\epsilon[H(\mu,p,\phi+\epsilon\psi)-H(\mu,p,\phi)]\\
	&= \int \partial_\alpha f(x,\mu,\phi(x))\psi(x)\mu(dx) -\int \nabla p(x) \partial_\alpha b(x,\mu,\phi(x))\psi(x) \mu(dx)\\
	&= \int \Bigl[ \partial_\alpha f(x,\mu,\phi(x))- \nabla p(x) \partial_\alpha b(x,\mu,\phi(x))\Bigr]\psi(x) \mu(dx).
\end{split}
\end{equation}
Putting together \eqref{fo:dJ} and \eqref{fo:delta_Hamiltonian} we get the desired result.
\end{proof}

So, at least informally, solving the ergodic McKean-Vlasov control problem reduces to the solution of the system:
\begin{equation}
\label{fo:optimality}
\begin{cases}
&0=\frac12 \Delta \nu -\mathrm{div} \bigl( b(\cdot,\nu,\phi(\cdot))\nu \bigr)\\
&0=\lambda+\frac12\Delta p(x) + b(x,\nu,\phi(x))\cdot\nabla p(x)-\int D_\mu f(\xi,\nu,\phi(\xi))(x)\, \nu(d\xi) - f(x,\nu,\phi(x))\\
&0= \partial_\alpha f(x,\nu,\phi(x)) - \nabla p(x) \partial_\alpha b(x,\nu,\phi(x)).
\end{cases}
\end{equation}
Note that the third equation guarantees the criticality of the function $\alpha\mapsto  f(x,\nu,\alpha) - y \cdot b(x,\nu,\alpha)$, so if we define the minimized Hamiltonian $H^\star$ by:
\begin{equation}
\label{fo:H_star}
	H^\star(x,\mu,y) = \inf_{\alpha} \bigl(f(x,\mu,\alpha) - y \cdot b(x,\mu,\alpha) \bigr),
\end{equation}
the above system can be written as:
\begin{equation}
\label{fo:optimality-H}
\begin{cases}
	&0=\frac12 \Delta \nu + \mathrm{div} \bigl( \partial_{y} H^\star(\cdot, \nu, \nabla p(\cdot)) \nu\bigr)\\
	&0=\lambda + \frac12\Delta p(x) - H^\star(x, \nu, \nabla p(x)) - \int D_\mu H^\star(\xi,\nu,\phi(\xi))(x)\, \nu(d\xi) .
\end{cases}
\end{equation}
Both systems should be completed with appropriate boundary conditions when needed (like for example in the next subsection where we use periodic boundary conditions to analyze the system on the torus) and the following condition:
$$
	\int \nu(x)dx = 1, 
$$
to which we can add the normalization condition:
$$
\int p(x) dx = 0
$$
to guarantee uniqueness for $p$. Indeed, the above equations~\eqref{fo:optimality-H} can only determine $p$ up to an additive constant.

\subsection{A Class of Models Amenable to Numerical Computations}
In general, computing the invariant distribution solving~\eqref{fo:Poisson} for a given $\phi$ can be costly. For example, it can be estimated by solving the PDE or by using Monte Carlo simulations for the MKV dynamics~\eqref{fo:dynamics}, see e.g.~\cite{MR1370849}. To simplify the presentation and focus on the main ingredients of the method proposed here, we shall consider a setting in which the optimal invariant distribution as well as the optimal control can both be expressed directly in terms of an adjoint variable. 

From now on we assume that $k=d$ and:
\begin{equation}
\label{fo:assumption}
b(x,\alpha)=b_0\alpha+\nabla\tilde b(x),
\qquad\text{and}\qquad
f(x,\mu,\alpha)=\frac12|\alpha|^2+\int \tilde f(x,\xi)\mu(d\xi)
\end{equation}
for a constant $b_0$ and functions $\tilde b$ and $\tilde f$ satisfying the following assumption.

\vskip 6pt 
\textbf{Assumption: } $\tilde b$ is of class $\cC^1$, and $\nabla \tilde b$ and $\tilde f$ are Lipschitz in both variables.

\begin{remark}
The dependence on the measure $\mu$ of the functions $f$ identified in \eqref{fo:assumption} is linear. However, the 
derivations and the results in the remainder of the paper extend easily to more general classes of dependence.
For example, similar proofs can be applied to functions $f$ of the form
$$
f(x,\mu,\alpha)=\frac12|\alpha|^2+\int \tilde f(x,\xi,\zeta)\mu(d\xi)\mu(d\zeta).
$$
In this case, the cost \eqref{fo:ergodic_cost_3} will be a triple integral of the same form which can be analyzed in the same way, even if the numerical computations may be slower.
\vskip 2pt
So it should be clear that many more functions $f$ can be handled as long as the functional derivative $D_\mu f$ can be computed and the cost \eqref{fo:ergodic_cost_3} can be expressed as the multiple integral with respect to the Lebesgue measure over the product of tori, of a function $\tilde F$ of the variables of the tori, $\nabla h(x)$, $\int e^{h(z)}dz$ and the function $e^h$ evaluated on the different tori. 
\end{remark}

\begin{remark}
Models with \emph{local interactions} have frequently been considered in the existing literature, and some of the numerical examples presented in Section \ref{sec:num-res} do involve local interactions. In these models, the function $f$ is of the form $f(x,\mu,\alpha)=\hat f(x,\mu(x),\alpha)$ where $\mu(x)$ represents the value of the density of the measure $\mu$ at the point $x$. In most cases, these models are more difficult to study analytically. However in the present situation, they represent a simplification when compared to the models with non local interactions studied here. Indeed, because of the special form \eqref{fo:gibbs} of the measures we work with, the fact that $f$ depends only upon the density of $\mu$ at the point $x$ streamlines the formula \eqref{fo:ergodic_cost_3} giving the cost $\tilde J(h)$ which is now expressed as a single integral over the variable $x$ only.
\end{remark}

\vskip 6pt\noindent
Let $(\nu^\star,p^\star,\phi^\star)$ be a solution to the optimality system \eqref{fo:optimality}. In this setting, the third equation of~\eqref{fo:optimality} gives $\phi^\star(x)=b_0\nabla p^\star(x)$. Substituting in the expression for the drift $b$ and the feedback function $\phi^\star$ into the state equation \eqref{fo:dynamics},
we see that at the optimum, we are dealing with a \emph{gradient diffusion}:
$$
dX_t=\frac12\nabla h^\star(X_t)dt +dW_t
$$
for the function 
\begin{equation}
	\label{ergo:opt-h}
	h^\star(x)=2\bigl(b_0^2 p^\star(x) + \tilde b(x)\bigr).
\end{equation}
  Accordingly, the invariant measure is necessarily of the form:
\begin{equation}
\label{fo:gibbs}
\nu^\star(x)=\frac{e^{h^\star(x)}}{\int e^{h^\star(x)}dx}.
\end{equation}
 Notice that in this case, the optimal control is given by 
 \begin{equation}
 \label{fo:phi_star}
 \phi^\star(x) = b_0 \nabla p^\star(x) = \frac{1}{b_0}(\frac{1}{2} \nabla h^\star(x) - \nabla \tilde b(x)).
 \end{equation}
Hence minimizing the ergodic cost~\eqref{fo:ergodic_cost_2} over controls $\phi$ under the constraint coming from the Poisson equation~\eqref{fo:Poisson} can be rephrased as the problem of minimizing the functional:
\begin{equation}
\label{fo:ergodic_cost_3}
\begin{split}
	\widetilde J(h)
	&= \int \int \widetilde F\left(x,y, \nabla h(x), e^{h(x)}, e^{h(y)}, \int e^{h(z)}dz\right) \; dx dy,
\end{split}
\end{equation}
over functions $h$, where:
\begin{equation}
\label{eq:ergo-def-tildeF-cost}
	\widetilde F(x,y, q, X, Y, Z) = \check f\left(x, y, q \right) X Y Z^{-2}
\end{equation}
with:
$$
	\check f (x,y,q) = \frac{1}{2}\Bigl| \frac{1}{b_0}\bigl(\frac{1}{2} q - \nabla \tilde b(x)\bigr)\Bigr|^2 + \tilde f(x,y).
$$
Notice that $F(x,y, q, \alpha X, \alpha Y, \alpha Z) =F(x,y, q, X, Y, Z) $ for all non-zero $\alpha\in\RR$. Hence, $\widetilde J(h+c)=\widetilde J(h)$ for every constant $c\in\RR$. So if $h^\star$ minimizes $\widetilde J$, so does $h^\star + c$ for any constant $c \in \RR$. This is not an issue to guarantee that the optimal value of $\widetilde J$ is close to the optimal value of the original cost $J$, but it can lead to numerical difficulties. For this reason it is possible to add a term of the form $|\int \tilde h|$  in~\eqref{fo:ergodic_cost_3}  in order to enforce a normalization condition and hence, uniqueness of the minimizer. The analysis presented below could be adapted to take into account this extra term at the expense of more cumbersome notations, so for the sake of clarity we only consider~\eqref{fo:ergodic_cost_3}. 

From this point on, instead of attempting to solve the system \eqref{fo:optimality} numerically, we search for the right function $h$ in a family of functions $x\mapsto h_\theta(x)$ parameterized by the parameter $\theta\in\Theta$. The desired parameter $\theta^\star$ minimizes the functional:
$$
	\JJ(\theta)=\int\int \FF_\theta(x,y) dx dy
$$
where: 
$$
	\FF_\theta(x,y)= \widetilde F\Bigl(x,y,\nabla h_\theta(x),e^{h_\theta(x)}, e^{h_\theta(y)},{\textstyle\int} e^{h_\theta(z)}dz\Bigr)
	=  \check f (x,y,\nabla h_\theta(x))\frac{e^{h_\theta(x)}}{\int e^{h_\theta(x)}dx}\frac{e^{h_\theta(y)}}{\int e^{h_\theta(y)}dy}.
$$
One should think of the function $h_\theta(\cdot)$ as a computable approximation of $h(\cdot)=2\bigl(b_0^2 p(\cdot) + \tilde b(\cdot)\bigr)$, allowing us to replace the minimization of the ergodic cost \eqref{fo:ergodic_cost} by the minimization:
\begin{equation}
\label{fo:NN_minimization}
\inf_{\theta\in\Theta}\int\int \FF_\theta(x,y) dx dy.
\end{equation}
Notice that the gradient (with respect to the parameter $\theta$) of $\FF$ can easily be computed. It reads:
$$
\partial_\theta \FF_\theta(x,y) = \FF_\theta(x,y)\Bigl(
\partial_\theta h_\theta(x)+\partial_\theta h_\theta(y)-2\frac{\int\partial_\theta h_\theta e^{h_\theta}}{\int e^{h_\theta}}
+\partial_\theta \log \check f\bigl(x,y,\nabla h_\theta(x)\bigr)
\Bigr).
$$

In anticipation of the set-up of next section where we consider our optimization problem on the torus, the double integral appearing in \eqref{fo:NN_minimization} can be viewed as an expectation, and its minimization is \emph{screaming} for the use of the Robbins-Monro procedure. Moreover, if we use the family $(h_\theta)_{\theta\in\Theta}$ given by a feed-forward neural network, this minimization can be implemented efficiently with the powerful tools based on the so-called Stochastic Gradient Descent (SGD) developed for the purpose of machine learning.

\section{A First Machine Learning Algorithm}
\label{sub:periodic_numerics}
We now restrict ourselves to the case of the torus $\TT^d = [0, 2 \pi]^d$ for the purpose of numerical computations. The admissible feedback functions $\phi$ being continuous, the drift $\TT^d\ni x\mapsto b(x,\phi(x))\in\RR^d$ is a bounded continuous function on the torus and the controlled state process $\bX=(X_t)_{t\ge 0}$ is a Markov process with infinitesimal generator:
$$
L=\frac12\Delta +b\bigl(\cdot,\phi(\cdot)\bigr)\nabla.
$$ 
The compactness of the state space $\TT^d$ and the uniform ellipticity of this generator guarantee that this state process is ergodic and that its invariant probability measure $\nu^\phi$ has a $\cC^\infty$ density with respect to the Riemannian measure on $\TT^d$ (which we assumed to be normalized to have total mass $1$). Note that, because the torus $\TT^d$ does not have a boundary, the integration by parts which we used freely in the computations of the above subsection are fully justified in the present situation.

\vskip 6pt
We introduce new notation to define the class of functions $(h_\theta)_{\theta\in\Theta}$ which we use for numerical approximation purposes.
We denote by:
$$
	\mathbf{L}^\psi_{d_1, d_2} = \left\{ \phi: \RR^{d_1} \to \RR^{d_2} \,\Big|\,  \exists \beta \in \RR^{d_2}, \exists w \in \RR^{d_2 \times d_1}, \forall  i \in \{1,\dots,d_2\}, \;\phi(x)_i = \psi\left(\beta_i + \sum_{j=1}^{d_1} w_{i,j} x_j\right) \right\} 
$$
the set of layer functions with input dimension $d_1$, output dimension $d_2$, and activation function $\psi: \RR \to \RR$. Here, $\circ$ denotes the composition of functions.
Building on this notation we define:
$$
	\bN^\psi_{d_0, \dots, d_{\ell+1}} 
	= 
	\left\{ \varphi: \RR^{d_0} \to \RR^{d_{\ell+1}} \,\Big|\, \forall i \in \{0, \dots, \ell-1\}, \exists \phi_i \in \mathbf{L}^\psi_{d_i, d_{i+1}}, \exists \phi_\ell \in \mathbf{L}_{d_{\ell}, d_{\ell+1}}, \varphi = \phi_\ell \circ \phi_{\ell-1} \circ \dots \circ \phi_0  \right\} \, 
$$
 the set of regression neural networks with $\ell$ hidden layers and one output layer, the activation function of the output layer being the identity $\psi(x)=x$. Note that as a general rule, we shall not use the superscript $\psi$ in that case. The number $\ell$ of hidden layers, the numbers $d_0$, $d_1$, $\cdots$ , $d_{\ell+1}$ of units per layer, and the activation functions (one single function $\psi$ in the present situation), are what is usually called the architecture of the network. Once it is fixed, the actual network function $\varphi\in \bN^\psi_{d_0, \dots, d_{\ell+1}} $ is determined by the remaining parameters:
 $$
 \theta=(\beta^{(0)}, w^{(0)},\beta^{(1)}, w^{(1)},\cdots\cdots,\beta^{(\ell-1)}, w^{(\ell-1)},\beta^{(\ell)}, w^{(\ell)})
 $$
defining the functions $\phi_0$, $\phi_1$, $\cdots$ , $\phi_{\ell-1}$ and $\phi_\ell$ respectively. Their set is denoted by $\Theta$. For each $\theta\in\Theta$, the function $\varphi$ computed by the network will be denoted by $h_\theta$.
As it should be clear from the discussion of the previous section, here, we are interested in the  case where $d_0 = d$ and $d_{\ell+1} = 1$.

Our analysis is based on the following algorithm. In practice, instead of having a fixed number of iterations $M$, one can use a criterion of the form: at iteration $m$, if $|\nabla \JJ_{S}(\theta_{m})|$ is small enough, stop; otherwise continue.

\begin{algorithm}
\DontPrintSemicolon
\KwData{An initial parameter $\theta_0\in\Theta$.  A sequence $(\alpha_m)_{m \geq 0}$ of learning rates.}
\KwResult{Parameters $\theta^\star$ such that $h_{\theta^\star}$ approximates $h^\star$}
\Begin{
  \For{$m = 0, 1, 2, \dots, M $}{
    Pick $S = (x_\ell, y_\ell)_{\ell=1}^L$ where $x_\ell$ and $y_\ell$ are picked i.i.d. uniform in $[0, 2 \pi]$\;
    Compute the gradient $\nabla \JJ_{S}(\theta_{m})$  of
    $
    	\JJ_{S}(\theta_{m}) = \frac1{L}\sum_{\ell=1}^L \FF_{\theta_{m}}(x_\ell, y_\ell)
    $ \;
    Set $\theta_{m+1} = \theta_{m} + \alpha_m \nabla \JJ_{S}(\theta_{m})$ \;
    }
  \KwRet{ $\theta_{M}$}
  }
\caption{SGD for ergodic MFC\label{algo:SGD-ergoMFC}}
\end{algorithm}

\subsection{The Approximation Estimates}
Our goal is now to analyze the error made by the numerical procedure described in Algorithm~\ref{algo:SGD-ergoMFC}. We split the error into two parts: the \emph{approximation error} and the \emph{estimation error} (or \emph{generalization error}). The approximation error quantifies the error made by shrinking the class of admissible controls (here we use neural networks of a certain architecture instead of all possible feedback controls). The estimation error quantifies the error made by replacing the integrals by averages over a finite number of Monte Carlo samples.

In this section, $\sigma:\RR \to \RR$ denotes a $2\pi-$periodic activation function of class $\cC^1$ whose Fourier expansion contains $1$, i.e.,
\begin{equation}
\label{eq:cond-sigma-Fourier}
	1 \in \left\{k \in \ZZ \,\Big|\, \hat \sigma(k) := \frac{1}{2 \pi} \int_{-\pi}^\pi \sigma(x) e^{-i k x} dx \neq 0 \right\}.
\end{equation}
More general activation functions (such as the hyperbolic tangent) could probably be considered at the expense of additional technicalities. The choice of this class of activation functions is motivated by the fact that we will use it to build a neural network which can approximate a periodic function together with its first order derivatives (namely, $h^\star$ and $\nabla h^*$).

\vskip 6pt
\noindent\textbf{Approximation Error. } 
The proof of our first estimate is based on the following special case of~\cite[Theorems 2.3 and 6.1]{MhaskarMicchelli}.\footnote{We use a special case of the \emph{neural networks} considered in~\cite{MhaskarMicchelli}. For us, using the notation of Mhaskar and Micchelli, $m_n = \hat\sigma(1)$ and $N_n = 1$ for all $n$.} We state it for the sake of completeness.
It provides a neural network approximation for a function and its derivative. 
For positive integers $n$ and $m$, and a function $g \in \cC(\TT^m) $, let $E_n^m(g)$ denote the trigonometric degree of approximation of $g$ defined by:
$$
	E_n^m(g) = \inf_{T} \|g - T\|_{\cC(\TT^m)}
$$
where the infimum is over trigonometric polynomials of degree at most $n$ in each of its $m$ variables.

\begin {theorem}[Theorems 2.3 and 6.1 in~\cite{MhaskarMicchelli}]
\label{th:MM}
	Let $f: \TT^d \to \RR$ be of class $\cC^2$, and let $n$ and $N$ be positive integers. Then there exist $n_{\mathrm{in}} \in O(N n^d)$ and $\varphi_f \in \bN^\sigma_{d, n_{\mathrm{in}}, 1}$ such that:
	\begin{align}
		\label{eq:MhaskarMicchelli-bdd-f}
		\|f - \varphi_f\|_{\cC(\TT^d)} &\leq c\left[E^d_n(f) + E^1_N( \sigma) n^{d/2} \|f\|_{\cC(\TT^d)} \right],
		\\
		\label{eq:MhaskarMicchelli-bdd-grad-f}
		\|\partial_i f - \partial_i \varphi_f\|_{\cC(\TT^d)} &\leq c\left[E^d_n(\partial_i f) + E^1_N(\sigma') n^{d/2} \|\partial_i f\|_{\cC(\TT^d)} \right], \quad i = 1,\dots,d,
	\end{align}
	where $c$ depends on the activation function through $\hat\sigma(1)$ but does not depend on $n,N, n_{\mathrm{in}}$.
\end{theorem}

The workhorse of our control of the approximation error is the following.

\begin{theorem} 
\label{thm:ergo-approx-NN}
Assume that for some integer $K \geq 1$, there exists a minimizer over $\cC^{K+1}(\TT^d)$, say $h^\star$, of the cost function $\widetilde J$ defined in \eqref{fo:ergodic_cost_3}. Assume that $\sigma \in \cC^{K+1}(\TT^d)$. Then, for $n_{\mathrm{in}}$ large enough we have:
	$$
		\widetilde J(h^\star)
		\leq 
		\inf_{\varphi \in \bN^\sigma_{d, n_{\mathrm{in}}, 1} } \widetilde J(\varphi)
		\leq
		\widetilde J(h^\star) + \epsilon_1(n_{\mathrm{in}})
	$$
	where 
	$$
		\epsilon_1(n_{in}) \in O\left((n_{in})^{-K/(3d)}\right).
	$$
The constants in the above big Bachmann - Landau term $O(\cdot)$ depend only on the data of the problem as well as $\hat \sigma(1)$, $K$, and the $\cC^0-$norms of the partial derivatives of $\sigma$ and $h^\star$ of order up to $K+1$ (but they do not depend upon $n_{\mathrm{in}}$).
\end{theorem}

\begin{remark}
The exponent in the $O$ term in the statement of the proposition is what is blamed for the so-called curse of dimensionality. In some settings, the constants in the $O$ term can be estimated if bounds on the $\cC^0-$norms of the partial derivatives of $h^\star$ of order up to $K+1$ are known, for instance from a priori estimates on the solution of the PDE system~\eqref{fo:optimality-H}.
\end{remark}

\begin{proof}[Proof of Theorem \ref{thm:ergo-approx-NN}]
The first inequality holds by definition of $h^\star$ and because $\bN^\sigma_{d, n_{\mathrm{in}}, 1} \subset \cC^{K+1}(\TT^d)$. 
In order to apply the result of Theorem~\ref{th:MM} to $f = h^\star$, we bound from above the right hand sides of~\eqref{eq:MhaskarMicchelli-bdd-f} and~\eqref{eq:MhaskarMicchelli-bdd-grad-f}. We use the fact that the trigonometric degree of approximation of a function of class $\cC^r$ is of order $O(n^{-r})$ when using polynomials of degree at most $n$. More precisely, by~\cite[Theorem 4.3]{MR0251410}, if $f:\RR^d \to \RR$ is an $r$-times continuously differentiable function which is $2\pi$-periodic in each variable, then for every positive integer $n$, there exists a trigonometric polynomial $T_{n}$ of degree at most $n$ such that 
$$
	|f(x) - T_{n}(x)| \leq C n^{-r}, \qquad x \in [0, 2\pi]^d,
$$
where $C$ depends only on $r$ and on the bounds on the $r$-th derivatives of $f$ in each direction: $\|\partial^{(r)}_i f\|_{\cC^0}$, $i=1,\dots,d$. We apply this result, with some integer $n$ to be specified later, to $f = h^\star$ and $f = \nabla h^\star$ with $r = K$, since $h^\star$ is of class $\cC^{K+1}$. 
By~\cite[Theorem 4.3]{MR0251410} again, since $\sigma$ and $\sigma'$ are both of class $\cC^K$, we obtain that for any integer $N$ there exist trigonometric polynomials $T_{N}, \tilde T_{N}$ of degree at most $N$ such that 
$$
	|\sigma(x) - T_{N}(x)| \leq C N^{-r}, \quad |\sigma'(x) - \tilde T_{N}(x)| \leq C N^{-r}, \qquad  0 \le x \le 2 \pi,
$$
where $C$ depends only on $r$ and on the bounds on the $K$-th derivatives of $\sigma$ and $\sigma'$, namely $\sigma^{(K)}, \sigma^{(K+1)}$. We apply this result with $N = n^{1+d/(2K)}$. Note that $N^{-r} = n^{-K - d/2}$.

So by Theorem~\ref{th:MM},  we obtain that
there exists $n_{\mathrm{in}} \in O(N n^d)$ and $\varphi_{h^\star} \in \bN^\sigma_{d, n_{\mathrm{in}}, 1}$ such that 
	\begin{align}
	\label{eq:approx-hstar-tmp}
		\|h^\star - \varphi_{h^\star}\|_{\cC(\TT^d)} + \|\nabla h^\star - \nabla \varphi_{h^\star}\|_{\cC(\TT^d)} &\leq C_1 n^{-K},
	\end{align}
	where the constant $C_1$  depends on $K, d, \|\partial^{(K)}_{i} h^\star\|_{\cC(\TT^d)}, \|\partial^{(K+1)}_{i}  h^\star\|_{\cC(\TT^d)}$, $i=1,\dots,d$, $\| \sigma^{(K)}  \|_{\cC(\RR)}, \| \sigma^{(K+1)} \|_{\cC(\RR)}$, but not on $n$ or $n_{\mathrm{in}}$. This implies in particular (since $n,K \geq 1$) that 
	$$
	\|\varphi_{h^\star}\|_{\cC(\TT^d)} + \|\nabla \varphi_{h^\star}\|_{\cC(\TT^d)}
	\leq C_2 := \|h^\star\|_{\cC(\TT^d)} + \|\nabla h^\star\|_{\cC(\TT^d)} + C_1.
	$$ 
	Notice that, since $Nn^d = n^{1+d/(2K) + d} = n^{\frac{Kd + K + d/2}{K}} \leq n^{3d}$, the number of units in the hidden layer is $n_{\mathrm{in}} \in O(n^{3d})$, hence the right hand side in~\eqref{eq:approx-hstar-tmp} is of order $O(n_{\mathrm{in}}^{-K/(3d)})$. In other words,
	\begin{align}
	\label{eq:approx-hstar-tmp-2}
		\|h^\star - \varphi_{h^\star}\|_{\cC(\TT^d)} + \|\nabla h^\star - \nabla \varphi_{h^\star}\|_{\cC(\TT^d)} &\in O\left(n_{\mathrm{in}}^{-K/(3d)}\right),
	\end{align}
	where the constant in the big $O$ might depend on $C_1$ but is independent of $n_{\mathrm{in}}$.

Going back to the definition~\eqref{fo:ergodic_cost_3} of $\widetilde J$, we note that, by~\eqref{eq:approx-hstar-tmp}, for all $x \in \TT^d$, $h^\star(x)$ and $\varphi_{h^\star}(x)$ both lie in the interval $[-C_2, C_2]$. Since $x \mapsto e^x$ is Lipschitz continuous on this interval with a Lipschitz constant depending only on $C_2$, we obtain that:
$$
	|e^{h^\star(x)} - e^{\varphi_{h^\star}(x)}| \leq c \|h^\star - \varphi_{h^\star}\|_{\cC(\TT^d)}, \qquad x \in \TT^d,
$$
where here and thereafter $c$ denotes a generic constant which depends on the data of the problem as well as $K$, and bounds on the $\cC^0-$norms of the partial derivatives of $\sigma$ and $h^\star$ up to order $K+1$, and whose exact value might change from line to line. 
Moreover, $\int e^{h^\star}, \int e^{\varphi_{h^\star}}$ lie in the interval $\left[e^{-C_2}, e^{C_2}\right]$, and $x \mapsto x^{-2}$ is Lipschitz continuous on this interval with a Lipschitz constant depending only on $C_2$ so we also have:
$$
	\Bigl|\Bigl({\textstyle\int} e^{h^\star}\Bigr)^{-2} - \Bigl({\textstyle\int} e^{\varphi_{h^\star}}\Bigr)^{-2} \Bigr| \leq c \|h^\star - \varphi_{h^\star}\|_{\cC(\TT^d)}.
$$
Hence, recalling the definition \eqref{eq:ergo-def-tildeF-cost} of $\widetilde F$, one can check after some calculations that for all $x,y \in \TT^d$,
\begin{align*}
	&\left| \widetilde F\left(x,y, \nabla h^\star(x), e^{h^\star(x)}, e^{h^\star(y)}, \int e^{h^\star(z)}dz\right) - \widetilde F\left(x,y, \nabla \varphi_{h^\star}(x), e^{\varphi_{h^\star}(x)}, e^{\varphi_{h^\star}(y)}, \int e^{\varphi_{h^\star}(z)}dz\right) \right|
	\\
	&\hskip 85pt
	\leq c \|h^\star - \varphi_{h^\star}\|_{\cC(\TT^d)}.
\end{align*}
From the definition~\eqref{fo:ergodic_cost_3} of $\widetilde J$, the considerations above and~\eqref{eq:approx-hstar-tmp-2}, we deduce that:
\begin{align}
\label{eq:final-bound-approx-directmethod}
	\left|\widetilde J(h^\star) - \widetilde J(\varphi_{h^\star})\right|
	&\leq c \|h^\star - \varphi_{h^\star}\|_{\cC(\TT^d)}
	\in O\left(n_{\mathrm{in}}^{-K/d}\right),
\end{align}
which completes the proof.
\end{proof}

\begin{corollary}
If $\sigma \in \cC^2$, $\tilde b \in \cC^2$ and if there exists a classical solution $(\nu^\star, p^\star, \lambda^\star)$ to the optimality system \eqref{fo:optimality-H}, then we have:
	$$
		\left| \inf_{h \in \cC^2}\widetilde J(h)
		-
		\inf_{\varphi \in \bN^\sigma_{d, n_{\mathrm{in}}, 1} } \widetilde J(\varphi) \right|
		\in
		O\left(n_{\mathrm{in}}^{-1/d}\right).
	$$
\end{corollary}

\begin{proof}
Indeed, if $\tilde b \in \cC^2$ and if we have existence of a classical solution $(\nu^\star, p^\star, \lambda^\star)$ to the optimality system \eqref{fo:optimality-H}, in particular if $p^\star \in \cC^2$, and $\|p^\star\|_{\cC(\TT^d)}$, $\|\partial_i p^\star\|_{\cC(\TT^d)}$ and $\|\partial_{i,j} p^\star\|_{\cC(\TT^d)}$, $i,j  \in \{1,\dots, d\}$, are bounded by constants depending only on the data of the problem, we obtain that $h^\star$ given by~\eqref{ergo:opt-h} provides a minimizer of $\widetilde J$ of class $\cC^2$. We can then apply Theorem~\ref{thm:ergo-approx-NN} with  $K=1$.
\end{proof}

\begin{remark}
For mean field games, existence of classical solutions to the ergodic PDE system has been studied in several settings, see e.g.~\cite{MR2269875,MR2295621}. To the best of our knowledge, corresponding results do not exist yet for the PDE system arising in the ergodic optimal control of MKV dynamics and this question will be addressed in a future work. In finite time horizon, existence of classical solutions has been studied e.g. in~\cite{MR3392611}.
\end{remark}

\vskip 6pt
\noindent\textbf{Estimation Error. } 
We then turn our attention to the estimation (or \emph{generalization}) error. Let $n_{\mathrm{in}}$ be a fixed positive integer. In the numerical implementation, we do not minimize directly $\widetilde J$ over a set of neural networks with say $n_{\mathrm{in}}$ units. Instead, we minimize over empirical versions computed from Monte Carlo samples. 
To be specific, for a given sample:
\begin{equation}
\label{fo:MC_sample}
S = ((x_\ell, y_\ell)_{\ell=1,\dots,L}, (z_q)_{q=1,\dots,Q}) \in (\TT^d \times \TT^d)^L \times (\TT^d)^Q
\end{equation} 
for which the $x_\ell, y_\ell, z_q$ are picked independently and uniformly in $[0, 2\pi]$, we minimize:
\begin{equation}
\label{fo:ergodic_cost_empirical}
\begin{split}
	\widehat J_S(h)
	&= \frac{1}{L} \sum_{\ell=1}^L  \widetilde F\left(x_\ell, y_\ell, \nabla h(x_\ell), e^{h(x_\ell)}, e^{h(y_\ell)},  \frac{1}{Q} \sum_{q=1}^Q e^{h(z_q)}  \right),
\end{split}
\end{equation}
where $\widetilde F$ is defined by~\eqref{eq:ergo-def-tildeF-cost}. The intuition is to approximate the double integral over $dx dy$ by an average over $L$ independent Monte Carlo samples $(x_\ell,y_\ell)$,  and likewise, the integral $\int e^{\varphi}$ by an empirical average over a sample of points uniformly distributed over $\TT^d$. 

\begin{remark}
The reader may be concerned by the slow rate of convergence of the Monte Carlo approximation. Indeed, this could be an issue as the convergence is only of the order $O(L^{-1/2})$. In fact, there are numerical approximations of the integrals which converge faster, and for which proven rates of convergence involve the number $L$ of sample points as well as the dimension $d$ of the torus. Many  Quasi Monte Carlo methods based on low discrepancy sequences (as opposed to independent identically distributed uniform samples as in the case of plain Monte Carlo computations) provide such improved rates of convergence. Depending upon the choice of the quasi Monte Carlo method, we can guarantee rates of convergence of the order $(\log L)^d/L$. See for example the excellent survey \cite{Dick_et_al}  for details and comparisons between various methods to compute integrals with respect to the Lebesgue measure over high dimensional cubes.  While the rate of convergence $O(L^{-1/2})$ of the classical Monte Carlo method is independent of the dimension, the advantage of these new rates of convergence is that not only do they guarantee faster convergence, but they also explicitly quantify how higher dimensions can slow down the convergence.  We chose to use plain Monte Carlo approximation procedures for the sake of simplicity.
\end{remark}

We shall use the following notation. For two positive constants $\gamma_1$ and $\gamma_2$, we denote by $ \bN^{\sigma,\gamma_1,\gamma_2}_{d, n_{\mathrm{in}}, 1}$ the set of functions $\varphi \in  \bN^{\sigma}_{d, n_{\mathrm{in}}, 1}$ for which there exist $(\bw,\bu) \in \RR^{n_{in}} \times \RR^{n_{in} \times (d+1)}$ satisfying 
\begin{itemize}
	\item $\|\bw\|_1 \leq \gamma_1$, 
	\item $\|\bu_n\|_2 \leq \gamma_2$ for all $n \in \{1,\dots,n_{in}\}$, 
\end{itemize}
and such that $\varphi(x) = \sum_{n = 1}^{n_{in}} w_n \sigma( \bu_n \cdot (x^\top, 1)^\top)$, where the superscript $\top$ denotes the transpose operation.
Here, $\|\bw\|_1 = \sum_{n=1}^{n_{in}} |w_n|$ and $\|\bu_n\|_2 = \sqrt{\sum_{k=1}^{d+1} |u_{n,k}|^2}$.

We are now in a position to prove the following bound on the uniform deviation between $\widetilde J$ and its empirical counterpart $\widehat J$. It is our main insight into the estimation error.

\begin{theorem}
\label{thm:ergo-estim-NN}
	Let $\gamma_1,\gamma_2$ be positive constants. We have:
	$$
		\EE_S\left[ \sup_{ \varphi \in \bN^{\sigma,\gamma_1,\gamma_2}_{d, n_{\mathrm{in}}, 1}} \left| \widetilde J(\varphi) - \widehat J_S(\varphi) \right| \right] \in O\left( \frac{1}{\sqrt{L}} + \frac{1}{\sqrt{Q}}\right),
	$$
where the expectation is over the samples $S$ as in~\eqref{fo:MC_sample}, and the constants in the big Bachmann - Landau term $O(\cdot)$ depend only on the data of the problem and on $\gamma_1$ and $\gamma_2$, but neither on $L$ nor on $Q$.
\end{theorem}
	
\begin{proof}
	First, introducing \emph{ghost} Monte Carlo samples $\tilde S = ((\tilde x_\ell, \tilde y_\ell)_\ell, (\tilde z_q)_{q})$ picked with the same distribution as $S$, and independent of the latter, we can rewrite \eqref{fo:ergodic_cost_3} as: 
	\begin{align*}
		\widetilde J(\varphi) = \EE_{\tilde S}\left[ \frac{1}{L} \sum_{\ell=1}^L \widetilde F\left(\tilde x_\ell, \tilde y_\ell, \nabla \varphi(\tilde x_\ell), e^{\varphi(\tilde x_\ell)}, e^{\varphi(\tilde y_\ell)}, \int e^{\varphi(z)}dz\right)\right],
	\end{align*}
where the expectation is over the samples $\tilde S$. Note that the variables $\tilde z_q$ do not appear in this expression. We kept them in the ghost sample for the sake of symmetry. Hence, for each fixed $S = ((x_\ell, y_\ell)_\ell, (z_q)_{q})$,
	\begin{align*}
\sup_{\varphi} \left| \widetilde J(\varphi) - \widehat J_S(\varphi) \right|
	&\leq \EE_{\tilde S} \sup_{\varphi}  \left| \frac{1}{L} \sum_{\ell=1}^L \left[ \widetilde F\left(\tilde x_\ell, \tilde y_\ell, \nabla \varphi(\tilde x_\ell), e^{\varphi(\tilde x_\ell)}, e^{\varphi(\tilde y_\ell)}, \int e^{\varphi(z)}dz\right) \right.\right.
	\\
	& \qquad\qquad 
	\left.\left. - \widetilde F\left(x_\ell, y_\ell, \nabla \varphi(x_\ell), e^{\varphi(x_\ell)}, e^{\varphi(y_\ell)}, \frac{1}{Q} \sum_{q=1}^Q e^{\varphi(z_q)} \right)\right]  \right|.
	\end{align*}
Taking  expectation over $S$ we get:
	
\begin{align}
\EE_S \sup_{\varphi} \left| \widetilde J(\varphi)   - \widehat J_S(\varphi) \right|
	&\leq \EE_{S,\tilde S} \sup_{\varphi}  \Bigl | \frac{1}{L} \sum_{\ell=1}^L  \Bigl[ \widetilde F\Bigl(\tilde x_\ell, \tilde y_\ell, \nabla \varphi(\tilde x_\ell), e^{\varphi(\tilde x_\ell)}, e^{\varphi(\tilde y_\ell)}, \int e^{\varphi(z)}dz\Bigr) 
	\notag
	\\
	& \qquad\qquad\qquad 
	- \widetilde F\Bigl(x_\ell, y_\ell, \nabla \varphi(x_\ell), e^{\varphi(x_\ell)}, e^{\varphi(y_\ell)},  \frac{1}{Q} \sum_{q=1}^Q e^{\varphi(z_q)} \Bigr)\Bigr]  \Bigr|
	\notag
	\\
&\leq \EE_{S,\tilde S} \sup_{\varphi}  \Bigl | \frac{1}{L} \sum_{\ell=1}^L  \Bigl[ \widetilde F\Bigl(\tilde x_\ell, \tilde y_\ell, \nabla \varphi(\tilde x_\ell), e^{\varphi(\tilde x_\ell)}, e^{\varphi(\tilde y_\ell)}, \int e^{\varphi(z)}dz\Bigr) 
	\notag
	\\
	& \qquad\qquad\qquad 
	- \widetilde F\Bigl(\tilde x_\ell, \tilde y_\ell, \nabla \varphi(\tilde x_\ell), e^{\varphi(\tilde x_\ell)}, e^{\varphi(\tilde y_\ell)},  \frac{1}{Q} \sum_{q=1}^Q e^{\varphi(\tilde z_q)} \Bigr)\Bigr]  \Bigr|
	\notag
	\\
	& \qquad\qquad
	+ \EE_{S,\tilde S} \sup_{\varphi}  \Bigl | \frac{1}{L} \sum_{\ell=1}^L  \Bigl[ \widetilde F\Bigl(\tilde x_\ell, \tilde y_\ell, \nabla \varphi(\tilde x_\ell), e^{\varphi(\tilde x_\ell)}, e^{\varphi(\tilde y_\ell)}, \frac{1}{Q} \sum_{q=1}^Q e^{\varphi(\tilde z_q)}\Bigr) 
	\notag
	\\
	& \qquad\qquad\qquad 
	- \widetilde F\Bigl(x_\ell, y_\ell, \nabla \varphi(x_\ell), e^{\varphi(x_\ell)}, e^{\varphi(y_\ell)},  \frac{1}{Q} \sum_{q=1}^Q e^{\varphi(z_q)} \Bigr)\Bigr]  \Bigr|
	\notag\\
&= (i) + (ii)
\label{fo:i+ii}
\end{align}
and we analyze separately the contributions of the two double expectations to the value of the above right hand side.
By definition of $\bN^{\sigma,\gamma_1,\gamma_2}_{d, n_{\mathrm{in}}, 1}$, there exists a constant $C'>0$ such that for every $\varphi\in\bN^{\sigma,\gamma_1,\gamma_2}_{d, n_{\mathrm{in}}, 1}$ and every sample $x,y, \tilde z_1, \dots, \tilde z_Q \in \TT^d$, we have:
$$
	\bigl(  x, y, \nabla \varphi(x), e^{\varphi(x)}, e^{\varphi(y)}\bigr) \in [-C',C']^{3d+2},
	\qquad  \int_{\TT^d} e^{\varphi(z)}dz\in  [\frac{1}{C'},C'], \qquad\frac{1}{Q} \sum_{q=1}^Q e^{\varphi(\tilde z_q)}\in  [\frac{1}{C'},C'],
$$
and given the assumptions on $\tilde f$ and the definitions of $\widetilde F$ and  $\bN^{\sigma,\gamma_1,\gamma_2}_{d, n_{\mathrm{in}}, 1}$, one can find a constant $C>0$ such that:
$$
|\widetilde F(x,y,q,X,Y,Z) - \widetilde F(x,y,q,X,Y,Z')|\le C|Z-Z'|
$$
for all $(x,y,q,X,Y)\in[-C',C']^{3d+2}$ and $Z$ and $Z'$ in $[1/C',C']$.
Notice that the constants $C$ and $C'$ depend upon $\sigma(\cdot)$, $\gamma_1$, $\gamma_2$, $d$, and $n_{\textrm{in}}$, but not on the particular $\varphi \in  \bN^{\sigma,\gamma_1,\gamma_2}_{d, n_{\mathrm{in}}, 1}$.
Using this Lipschitz bound, we get:
$$
(i)\le \EE_{\tilde S}\sup_\varphi\Bigl|\int e^{\varphi(z)}dz -\frac{1}{Q} \sum_{q=1}^Q e^{\varphi(\tilde z_q)} \Bigr|.
$$
To bound from above the right hand side, we follow a pretty standard strategy. First we notice that:
$$
\int e^{\varphi(z)}dz =\frac{1}{Q} \sum_{q=1}^Q\EE_S e^{\varphi( z_q)}.
$$
Moreover, we can introduce a family $\br=(r_q)_{q=1,\cdots,Q}$ of independent  Rademacher random variables (i.e. satisfying $\PP[ r_q=-1]=\PP[ r_q=1]=1/2$), independent of the samples $S$ and $\tilde S$. Since the samples $(z_q)_{q=1,\cdots,Q}$ and $(\tilde z_q)_{q=1,\cdots,Q}$ are independent and identically distributed, we have
\begin{align}
(i)&\le \EE_S\EE_{\tilde S}\sup_\varphi \Bigl|\frac{1}{Q} \sum_{q=1}^Q\Bigl[ e^{\varphi( z_q)}- e^{\varphi( \tilde z_q)}\Bigr]\Bigr|\notag\\
&= \EE_S\EE_{\tilde S}\EE_{\br}\sup_\varphi \Bigl|\frac{1}{Q} \sum_{q=1}^Q r_q\Bigl[ e^{\varphi( z_q)}- e^{\varphi( \tilde z_q)}\Bigr]\Bigr|\notag\\
&\le \EE_S\EE_{\tilde S}\EE_{\br}\sup_\varphi \Bigl|\frac{1}{Q} \sum_{q=1}^Q r_q e^{\varphi( z_q)}\Bigr|
+ \EE_S\EE_{\tilde S}\EE_{\br} \Bigl|\frac{1}{Q} \sum_{q=1}^Q r_qe^{\varphi( z_q)}\Bigr|\notag\\
&= 2\EE_S\EE_{\br}\sup_\varphi \Bigl|\frac{1}{Q} \sum_{q=1}^Q r_qe^{\varphi( z_q)}\Bigr|\notag\\
&\le C' \frac{1}{\sqrt{Q}},
\label{fo:i}
\end{align}
where we used Khintchine inequality to derive the last inequality.

\vskip 4pt
We now turn our attention to the estimation of the term $(ii)$ in \eqref{fo:i+ii}. Because of the introduction of the ghost samples, for each $\ell \in \{1,\cdots,L\}$, the two terms we compute the difference of are independent and identically distributed, so we can rewrite $(ii)$ using a family $\tau=(\tau_\ell)_{\ell=1,\cdots,L}$ of independent  Rademacher random variables independent of the samples $S$ and $\tilde S$ in the following way: 
\begin{align}
(ii)	&= \EE_{S,\tilde S} \EE_\tau \sup_{\varphi}  \Bigl | \frac{1}{L} \sum_{\ell=1}^L \tau_\ell\Bigl[ \widetilde F\Bigl(\tilde x_\ell, \tilde y_\ell, \nabla \varphi(\tilde x_\ell), e^{\varphi(\tilde x_\ell)}, e^{\varphi(\tilde y_\ell)}, \frac{1}{Q} \sum_{q=1}^Q e^{\varphi(\tilde z_q)}\Bigr) 
	\notag
	\\
	& \qquad\qquad\qquad 
	- \widetilde F\Bigl(x_\ell, y_\ell, \nabla \varphi(x_\ell), e^{\varphi(x_\ell)}, e^{\varphi(y_\ell)},  \frac{1}{Q} \sum_{q=1}^Q e^{\varphi(z_q)} \Bigr)\Bigr]  \Bigr|
	\notag\\
	&\leq \EE_{S,\tilde S} \EE_\tau \Bigl[ \sup_{\varphi}  \Bigl | \frac{1}{L} \sum_{\ell=1}^L \tau_\ell \widetilde F\Bigl(\tilde x_\ell, \tilde y_\ell, \nabla \varphi(\tilde x_\ell), e^{\varphi(\tilde x_\ell)}, e^{\varphi(\tilde y_\ell)},\frac{1}{Q} \sum_{q=1}^Q e^{\varphi(\tilde z_q)}\Bigr)  \Bigr|
	\notag
	\\
	& \qquad\qquad 
	+  \sup_{\varphi}  \Bigl | \frac{1}{L} \sum_{\ell=1}^L \tau_\ell \widetilde F\Bigl(x_\ell, y_\ell, \nabla \varphi(x_\ell), e^{\varphi(x_\ell)}, e^{\varphi(y_\ell)}, \frac{1}{Q} \sum_{q=1}^Q e^{\varphi(z_q)}  \Bigr)  \Bigr| \Bigr]
	\notag
	\\
	&=  2  \EE_S \EE_\tau \sup_{\varphi}  \Bigl | \frac{1}{L} \sum_{\ell=1}^L \tau_\ell \widetilde F\Bigl(  x_\ell,   y_\ell, \nabla \varphi(  x_\ell), e^{\varphi(  x_\ell)}, e^{\varphi(  y_\ell)}, \frac{1}{Q} \sum_{q=1}^Q e^{\varphi(z_q)} \Bigr)\Bigr|
	\label{fo:ii}
	\end{align}
where we used the fact that $S$ and $\tilde S$ are i.i.d. 
For each fixed $\varphi \in \bN^{\sigma,\gamma_1,\gamma_2}_{d, n_{\mathrm{in}}, 1}$ and samples $S$ and $\tau$ we have
	\begin{align*}
		&\left| \sum_{\ell=1}^L \tau_\ell \widetilde F\left(  x_\ell,   y_\ell, \nabla \varphi(  x_\ell), e^{\varphi(  x_\ell)}, e^{\varphi(  y_\ell)},  \frac{1}{Q} \sum_{q=1}^Q e^{\varphi(z_q)} \right) \right|
		\\
		&\hskip 35pt\leq 
		\left| \sum_{\ell=1}^L \tau_\ell \widetilde F\left(  0,  0, 0, 0, 0, 1\right) \right|
		\\
		& \hskip 55pt+ \left| \sum_{\ell=1}^L \tau_\ell \left[\widetilde F\left(  x_\ell,   y_\ell, \nabla \varphi(  x_\ell), e^{\varphi(  x_\ell)}, e^{\varphi(  y_\ell)}, \frac{1}{Q} \sum_{q=1}^Q e^{\varphi(z_q)} \right) - \widetilde F\left(  0, 0,0, 0, 0, 1 \right) \right] \right|.
	\end{align*}
Since $\widetilde F\left(  0,  0, 0, 0, 0, 1 \right)$ is a constant independent of $\varphi \in \bN^{\sigma,\gamma_1,\gamma_2}_{d, n_{\mathrm{in}}, 1}$, we denote it momentarily by $C''$ to ease the notation. Moreover, for each fixed sample $S = (x_\ell,   y_\ell, z_q)_{\ell,q}$ we have:	
	\begin{equation}
	\label{eq:proof-estimerr-FF0-tilde-z}
	\begin{split}
		&\EE_{\tau} \sup_{\varphi \in \bN^{\sigma,\gamma_1,\gamma_2}_{d, n_{\mathrm{in}}, 1}} \left| \sum_{\ell=1}^L \tau_\ell \left[\widetilde F\left(  x_\ell,   y_\ell, \nabla \varphi(  x_\ell), e^{\varphi(  x_\ell)}, e^{\varphi(  y_\ell)}, \frac{1}{Q} \sum_{q=1}^Q e^{\varphi(z_q)} \right) - C'' \right] \right|
		\\
		&\le \EE_{\tau} \sup_{\varphi \in \bN^{\sigma,\gamma_1,\gamma_2}_{d, n_{\mathrm{in}}, 1}} \sup_{s \in \{-1,+1\}} s \sum_{\ell=1}^L \tau_\ell \left[\widetilde F\left(  x_\ell,   y_\ell, \nabla \varphi(  x_\ell), e^{\varphi(  x_\ell)}, e^{\varphi(  y_\ell)}, \frac{1}{Q} \sum_{q=1}^Q e^{\varphi(z_q)} \right) - C''\right]
	\end{split}
	\end{equation}
where the variable $s$ is introduced to replace the absolute value.
For each $\varphi \in  \bN^{\sigma,\gamma_1,\gamma_2}_{d, n_{\mathrm{in}}, 1}$ and $s \in \{-1,+1\}$,  we define the function $\psi^s_\varphi$ for  $(x,y,Z)\in \TT^{d}\times\TT^{d}\times (\TT^{d})^Q$ by:
$$
	\psi^s_\varphi(x,y,(z_q)_{q=1,\dots,Q}) = \Bigl(  x, y, \nabla \varphi(x), e^{\varphi(x)}, e^{\varphi(y)}, \frac{1}{Q} \sum_{q=1}^Q e^{\varphi(z_q)}, s\Bigr). 
$$
By definition of $\bN^{\sigma,\gamma_1,\gamma_2}_{d, n_{\mathrm{in}}, 1}$, it is clear that the range of the map $\psi^s_\varphi$ is contained in a hypercube of the form: 
$$
	\cD^{\sigma, \gamma_1,\gamma_2}_{d, n_{\mathrm{in}}}=[-C',C']^{3d+2}\times [\frac{1}{C'},C']\times[-1,1]\subset \RR^{3d+4}
$$ 
where the constant $C'=C'(\sigma, \gamma_1,\gamma_2,d, n_{\textrm{in}})$ was introduced earlier.
Given the assumptions on $\tilde f$ and the definitions of $\widetilde F$ and  $\bN^{\sigma,\gamma_1,\gamma_2}_{d, n_{\mathrm{in}}, 1}$, one can construct a real valued function $ \Phi$ on $\RR^{3d+4}$ which is Lipschitz continuous over the whole space and satisfies:
$$
	 \Phi(x, y, q, X, Y, Z, s) = s \left[\widetilde F(x,y,q,X,Y,Z) - \widetilde F(0,0,0,0,0,1)\right], 
$$
for $(x, y, q, X, Y, Z, s)\in\cD^{\sigma, \gamma_1,\gamma_2}_{d, n_{\mathrm{in}}}$, and whose Lipschitz constant, say $K$, depends upon $\sigma(\cdot)$, $\gamma_1$, $\gamma_2$, $d$, $n_{\textrm{in}}$ and $C''$, but not on the particular $\varphi \in  \bN^{\sigma,\gamma_1,\gamma_2}_{d, n_{\mathrm{in}}, 1}$.
Next, we introduce the set of functions:
$$
	 \cF^{\sigma, \gamma_1,\gamma_2}_{d, n_{\mathrm{in}}} = \left\{  \psi : \TT^{d}\times\TT^d\times(\TT^d)^Q \to \RR^{3d+4}; \; \exists \varphi \in \bN^{\sigma,\gamma_1,\gamma_2}_{d, n_{\mathrm{in}}, 1},  \exists s \in \{-1,+1\},  \psi =  \psi^s_\varphi \right\}.
$$
We can then rewrite the right hand side of~\eqref{eq:proof-estimerr-FF0-tilde-z} as:
\begin{align*}
	\EE_{\tau} \sup_{ \psi \in \cF^{\sigma, \gamma_1,\gamma_2}_{d, n_{\mathrm{in}}}} \sum_{\ell=1}^L \tau_\ell \left[ \Phi \left( \psi \left(x_\ell,   y_\ell, (z_q)_q \right) \right) \right]
\end{align*}
By the form of Talagrand's contraction lemma given in Corollary 4 of \cite{Maurer},  this quantity is bounded from above by:
$$
\sqrt{2} K \EE_{\tilde\tau} \sup_{\psi \in \cF^{\sigma, \gamma_1,\gamma_2}_{d, n_{\mathrm{in}}}} \sum_{\ell=1}^L\sum_{k=1}^{3d+4} \tilde\tau_{\ell,k} \,  \psi \left(x_\ell,   y_\ell, (z_q)_{q=1,\dots,Q} \right)_k 
$$
where, for $\bz=(z_q)_{q=1,\dots,Q} \in (\TT^d)^Q$, $ \psi (x, y, \bz)_k$ denotes the $k$-th component of the vector $ \psi (x, y, \bz)$, and where the family of random variables $\tilde\tau=(\tilde\tau_{\ell,k})_{\ell=1,\cdots,L,\;k=1,\cdots,3d+4}$ is an independent Rademacher family with one extra index. Accordingly, this quantity is bounded from above by:
$$ 
\sqrt{2} K \sum_{k=1}^{3d+4}\EE_{\tilde\tau} \sup_{\psi \in \cF^{\sigma, \gamma_1,\gamma_2}_{d, n_{\mathrm{in}}}} \sum_{\ell=1}^L \tilde\tau_{\ell,k} \, \psi \left(x_\ell,   y_\ell, \bz \right)_k ,
$$
and we proceed to estimate the $3d+4$ terms of the outer sum one by one.
Notice that for $k \le 3d+2$, the term
$ \psi \left(x_\ell,   y_\ell, \bz \right)_k$
does not depend upon $\bz$, and we proceed in the following way.
The terms corresponding to $k=1,\dots,d$  (resp. $k=d+1,\dots,2d$) are easy to control since $\psi (x_\ell,   y_\ell,\bz)_k=(x_\ell)_k$ (resp. $\psi (x_\ell,   y_\ell,\bz)_k=(y_\ell)_{k-d}$) do not depend upon $\varphi$ or $s$, rendering the supremum irrelevant. Moreover, since the norms of $x_\ell $ and $y_\ell$ in $\RR^d$ are bounded by $C'$, Khintchine inequality gives:
$$
	\EE_{\tilde\tau} \sum_{\ell=1}^L \tilde\tau_{\ell,1} \, \|x_\ell\|
	\leq
	C \sqrt{L},
	\quad\text{and}\quad
\EE_{\tilde\tau} \sum_{\ell=1}^L \tilde\tau_{\ell,2} \, \|y_\ell\|
	\leq
	C \sqrt{L}.
$$
For $k=2d+h$ with  $h \in \{1,\dots,d\}$, $\psi (x_\ell,   y_\ell,\bz)_k=\partial_h\varphi(x_\ell)=\sum_{n=1}^{n_{in}} w_n u_{n,h} \sigma'( \bu_n \cdot  (x_\ell^\top, 1)^\top)$, and: 
\begin{align}
\label{fo:est}
	\sum_{\ell=1}^L\tilde\tau_{\ell,k} \, \Bigl[ \sum_{n=1}^{n_{in}} w_n u_{n,h} \sigma'( \bu_n \cdot  (x_\ell^\top, 1)^\top)\Bigr]
	&\leq
	\sum_{n=1}^{n_{in}} |w_n u_{n,h}| \Bigl|  \sum_{\ell=1}^L\tilde\tau_{\ell,k} \, \left[ \sigma'( \bu_n \cdot  (x_\ell^\top, 1)^\top)\right] \Bigr| \nonumber
	\\
	&\leq
	\Bigl(\sum_{n=1}^{n_{in}} |w_n u_{n,h}| \Bigr) \sup_{n}\Bigl|  \sum_{\ell=1}^L\tilde\tau_{\ell,k} \, \left[ \sigma'( \bu_n \cdot  (x_\ell^\top, 1)^\top)\right] \Bigr| \nonumber
	\\
	&\leq
	\gamma_1\gamma_2 \sup_{n} \Bigl|  \sum_{\ell=1}^L\tilde\tau_{\ell,k} \, \left[ \sigma'( \bu_n \cdot  (x_\ell^\top, 1)^\top)\right] \Bigr|
\end{align}
because of the definition of $\varphi \in  \bN^{\sigma,\gamma_1,\gamma_2}_{d, n_{\mathrm{in}}, 1}$. Consequently, since the above quantity does not depend upon $s$, the supremum over $\psi \in \cF^{\sigma, \gamma_1,\gamma_2}_{d, n_{\mathrm{in}}}$ can be taken over $\varphi \in \bN^{\sigma,\gamma_1,\gamma_2}_{d, n_{\mathrm{in}}, 1}$ and we have:
\begin{align*}
\EE_{\tilde\tau} \sup_{\psi \in \cF^{\sigma, \gamma_1,\gamma_2}_{d, n_{\mathrm{in}}}} \sum_{\ell=1}^L \tilde\tau_{\ell,k} \, \psi (x_\ell,   y_\ell,\bz)_k
	&=\EE_{\tilde\tau} \sup_{\varphi \in \bN^{\sigma,\gamma_1,\gamma_2}_{d, n_{\mathrm{in}}, 1}} \sum_{\ell=1}^L \tilde\tau_{\ell,k} \, \partial_h \varphi(x_\ell) 
	\\
	&=\EE_{\tilde\tau} \sup_{\varphi \in \bN^{\sigma,\gamma_1,\gamma_2}_{d, n_{\mathrm{in}}, 1}} \sum_{\ell=1}^L\tilde\tau_{\ell,k} \, \Bigl[ \sum_{n=1}^{n_{in}} w_n u_{n,h} \sigma'( \bu_n \cdot  (x_\ell^\top, 1)^\top)\Bigr]\\
	&\leq\gamma_1\gamma_2\EE_{\tilde\tau} \sup_{\bu_n \,:\, \|\bu_n\|_2 \leq \gamma_2} \, \Bigl|\sum_{\ell=1}^L \tilde\tau_{\ell,k} \,  \sigma'( \bu_n \cdot  (x_\ell^\top, 1)^\top)\Bigr|
\end{align*}
where we used \eqref{fo:est}. Since the derivative of the activation function $\sigma$ is Lipschitz (without any loss of generality we use the same constant $K>0$ for its Lipschitz constant), we can use the original version of Talagrand's contraction lemma to estimate the above right hand side. From \cite[Theorem 4.12]{MR1102015} with $F(x)=x$, $\varphi_i(t)=\sigma'(t)-\sigma'(0)$, and $T=\{(\bu\cdot(x_\ell^\top,1)^\top)_{\ell=1,\cdots,L};\;\|\bu\|_2\le\gamma_2\}$ we get:
\begin{align*}
	\EE_{\tilde\tau} \sup_{\psi \in \cF^{\sigma, \gamma_1,\gamma_2}_{d, n_{\mathrm{in}}}} \sum_{\ell=1}^L \tilde\tau_{\ell,k} \, \psi (x_\ell,   y_\ell,\bz)_k 
	&\leq C 
	\EE_{\tilde\tau} \sup_{\|\bu_n\|_2 \leq \gamma_2} \, \Bigl|\sum_{\ell=1}^L \tilde\tau_{\ell,k} [  \sigma'( \bu_n \cdot  (x_\ell^\top, 1)^\top)-\sigma'(0)]\Bigr|
	+C\sigma'(0)\EE_{\tilde\tau}  \, \Bigl|\sum_{\ell=1}^L \tilde\tau_{\ell,k} \Bigr|
	\\
	&\leq C \left(\EE_{\tilde\tau} \sup_{\|\bu\|_2 \leq \gamma_2} \, \Bigl| \sum_{\ell=1}^L \tilde\tau_{\ell,k} \, \left[  \bu \cdot (x_\ell^\top, 1)^\top \right] \Bigr|
	+\EE_{\tilde\tau} \, \Bigl|\sum_{\ell=1}^L \tilde\tau_{\ell,k} \Bigr|\right)
	\\
	&= C \left(\EE_{\tilde\tau} \sup_{\|\bu\|_2 \leq \gamma_2} \, \Bigl| \bu\cdot \sum_{\ell=1}^L\tilde\tau_{\ell,k} \, \left[   (x_\ell^\top, 1)^\top \right] \Bigr|
	+\EE_{\tilde\tau}  \, \Bigl|\sum_{\ell=1}^L \tilde\tau_{\ell,k} \Bigr|\right)
	\\
	&\le C \left(\EE_{\tilde\tau} \sup_{ \|\bu\|_2 \leq \gamma_2} \, \| \bu \|_2 \Bigl\| \sum_{\ell=1}^L\tilde\tau_{\ell,k} \, \left[   (x_\ell^\top, 1)^\top \right] \Bigr\|_2
	+\EE_{\tilde\tau} \, \Bigl|\sum_{\ell=1}^L \tilde\tau_{\ell,k} \Bigr|\right)
	\\
	&\leq C \left( \EE_{\tilde\tau}  \, \Bigl\| \sum_{\ell=1}^L\tilde\tau_{\ell,k} \, \left[   (x_\ell^\top, 1)^\top \right] \Bigr\|_2
	+\EE_{\tilde\tau}  \, \Bigl|\sum_{\ell=1}^L \tilde\tau_{\ell,k} \Bigr|\right)
	\\
	&\leq C \sqrt{L},
\end{align*}
where the value of the constant $C>0$ changed from line to line, and where we used Cauchy-Schwarz and Khintchine inequalities.

We proceed similarly for the values $k=3d+1$ and $k=3d+2$ since the exponential function is Lipschitz on the range $[-C',C']$ of the functions $\varphi \in \bN^{\sigma,\gamma_1,\gamma_2}_{d, n_{\mathrm{in}}, 1}$.

\vskip 2pt
We now focus on the penultimate term. For $k = 3d+3$, the term $\psi (x_\ell,   y_\ell,\bz)_k$ is $\frac{1}{Q} \sum_{q=1}^Q e^{\varphi(z_q)}$. It does not depend upon $s$ so the supremum over $\psi \in \cF^{\sigma, \gamma_1,\gamma_2}_{d, n_{\mathrm{in}}}$ can be taken over $\varphi \in \bN^{\sigma,\gamma_1,\gamma_2}_{d, n_{\mathrm{in}}}$ and we have:
\begin{align*}	
	\EE_{\tilde\tau} \sup_{ \psi \in \cF^{\sigma, \gamma_1,\gamma_2}_{d, n_{\mathrm{in}}}} \sum_{\ell=1}^L \tilde\tau_{\ell,k} \,  \psi (x_\ell,   y_\ell, \bz)_k
	&=\EE_{\tilde\tau} \sup_{\varphi \in \bN^{\sigma,\gamma_1,\gamma_2}_{d, n_{\mathrm{in}}, 1}} \sum_{\ell=1}^L \tilde\tau_{\ell,k} \, \frac{1}{Q} \sum_{q=1}^Q e^{\varphi(z_q)}
	\\
	&\leq
	C'\EE_{\tilde\tau}  \Bigl|\sum_{\ell=1}^L \tilde\tau_{\ell,k}  \Bigr|
	\\
	&\leq 
	C \sqrt{L} 
\end{align*}
because of Khintchine inequality.
Finally, the term corresponding to $k=3d+4$ can easily be bounded in the same way since $ \psi (x_\ell,   y_\ell, \bz)_{3d+4} = s\in\{-1,+1\}$.

This concludes the analysis of $(ii)$, proving that it is bounded from above by a constant times $1/\sqrt{L}$. Combining this with \eqref{fo:i} and \eqref{fo:i+ii}, the proof is complete.
\end{proof}

\section{Application of the Deep Galerkin Method}
\label{sec:DGM}
An alternative way to solve the ergodic mean field control problem is to tackle directly the PDE system~\eqref{fo:optimality-H}. In order to do so, we adapt the Deep Galerkin Method (DGM) proposed by Sirignano and Spiliopoulos~\cite{MR3874585} for a single PDE. The key idea is to rewrite the PDE system as a new minimization problem where the control is the triple $(\nu, p, \lambda)$ and the loss function is the sum of the PDE residuals (plus some terms taking into account the boundary conditions and the normalization conditions). In our setting, this idea can be implemented as follows. To alleviate the notations, we introduce the sets $C_{i} = \{x_i = 0\}$ where $i \in \{1,\dots,d\}$ and use the shorthand notation $(x_{-i}, 2\pi)$  for $(x_1,\dots,x_{i-1},2\pi,x_{i+1}, \dots, x_d)$ and we set
\begin{equation}
\label{eq:loss-DGM-total}
	L(\nu,p,\lambda) = L^{(1)}(\nu,p,\lambda) + L^{(2)}(\nu,p,\lambda)
\end{equation}
where
\begin{equation}
\label{eq:loss-DGM-1}
\begin{split}
	 L^{(1)}(\nu,p,\lambda)
	 &= \left\| \frac12 \Delta \nu + \mathrm{div} \bigl( \partial_{y} H^\star(\cdot, \nu, \nabla p(\cdot)) \nu \bigr) \right\|_{L^2(\TT^d)}
	 \\
	 & \qquad + \sum_{i=1}^d \left(\int_{C_i} \left| \nu(x) - \nu((x_{-i},2\pi)) \right|^2 dx\right)^{1/2}
	  + \left| 1 - \int_{\TT^d} \nu(x) dx \right|
\end{split}
\end{equation}
and 
\begin{equation}
\label{eq:loss-DGM-2}
\begin{split}
	 L^{(2)}(\nu,p,\lambda)
	&= \left\| \lambda + \frac12\Delta p - H^\star(\cdot, \nu, \nabla p(\cdot)) - \int D_\mu H^\star(\xi, \nu, \nabla p(\xi))(\cdot)\, \nu(d\xi) \right\|_{L^2(\TT^d)}
	 \\
	 & \qquad + \sum_{i=1}^d \left(\int_{C_i} \left| p(x) - p((x_{-i},2\pi)) \right|^2 dx\right)^{1/2}
	  + \left| \int_{\TT^d} p(x) dx \right|.
\end{split}
\end{equation}
Each function encodes one of the two PDEs of the optimality system~\eqref{fo:optimality-H} and contains one term for the residual of the PDE, one term for the periodicity condition, and one term for the normalization condition. These terms can be weighted to adjust their relative importance. In any case, note that $L(\nu,p,\lambda) = 0$ if $(\nu,p,\lambda)$ solves the PDE system~\eqref{fo:optimality-H}. 
Since our primary motivation is the optimal control of MKV dynamics, we present the method in this setting. However the same ideas can be readily applied to other PDE systems by designing differently the loss function. For instance,  to solve the PDE system arising in the corresponding stationary MFG, one simply needs to remove the term $\int D_\mu H^\star(\xi, \nu, \nabla p(\xi))(\cdot)\, \nu(d\xi)$ in~\eqref{eq:loss-DGM-2}. For the sake of illustration, we present several examples below in the next section.

\vskip 6pt
We then look for $\nu$ and $p$ in the form of neural networks, say $\nu_{\theta_1}$ and $p_{\theta_2}$ with fixed architectures and parameterized by $\theta_1$ and $\theta_2$ respectively. The unknown $\lambda$ is replaced by a variable coefficient $\theta_3 \in \RR$ which is learnt along the way. As in the method discussed in the previous sections, the integrals are interpreted as expectations with respect to a random variable with uniform distribution over $\TT^d$, and one uses SGD to minimize the total loss function. More precisely, for a given $\mathbf{S} = (S, (S_i)_{i \in \{1,\dots,d\}})$ where $S \subset [0,1]$ is a  finite set of points and $S_{i}$ is a finite set of points in $C_{i}$ for every $i \in \{1,\dots,d\}$, we define the empirical loss function as follows: for $\theta = (\theta_1,\theta_2,\theta_3)$,
\begin{equation}
\label{eq:loss-DGM-system}
	L_{\mathbf{S}}(\theta) = 
	L_{\mathbf{S}}^{(1)}(\theta) + L_{\mathbf{S}}^{(2)}(\theta)
\end{equation}
where
\begin{align*}
	L_{\mathbf{S}}^{(1)}(\theta) 
	= \, & \left(\sum_{x \in S} \left | \frac12 \Delta \nu_{\theta_1}(x) + \mathrm{div} \bigl( \partial_{y} H^\star(\cdot, \nu_{\theta_1}, \nabla p_{\theta_2}(\cdot)) \nu_{\theta_1} \bigr)(x) \right |^2\right)^{1/2}
	\\
	&\qquad + \sum_{i=1}^d \left(\sum_{x \in S_i} \left | \nu_{\theta_1}(x) - \nu_{\theta_1}((x_{-i},2\pi))\right |^2\right)^{1/2}
	+ \left| 1 - \sum_{x \in S} \nu_{\theta_1}(x) \right|
\end{align*}
and
\begin{align*}
	L_{\mathbf{S}}^{(2)}(\theta) 
	= \, & \left(\sum_{x \in S} \left | \theta_3 + \frac12\Delta p_{\theta_2}(x) - H^\star(x, \nu_{\theta_1}, \nabla p_{\theta_2}(x)) - \int D_\mu H^\star(\xi,\nu_{\theta_1}, \nabla p_{\theta_2}(\xi))(x)\, \nu_{\theta_1}(d\xi) \right |^2 \right)^{1/2}
	\\
	&\qquad + \sum_{i=1}^d \left( \sum_{x \in S_i} \left | p_{\theta_2}(x) - p_{\theta_2}((x_{-i},2\pi))\right |^2 \right)^{1/2}
	+ \left| \sum_{x \in S} p_{\theta_2}(x)\right|.
\end{align*}

One can use SGD to minimize the loss function~\eqref{eq:loss-DGM-system}. The approximation power of this method has been discussed in~\cite{MR3874585} using a universal approximation theorem. However, this type of results does not give any rate of convergence. More precise convergence results could be obtained by the techniques presented in Section~\ref{sub:periodic_numerics}. In particular, the approximation error can be bounded by combining again Theorems 2.3 and 6.1 in~\cite{MhaskarMicchelli}. For instance, if the PDE system~\eqref{fo:optimality-H} has a solution $(\nu, p, \lambda)$ such that $p, \nu \in \mathcal C^{3}(\TT^d)$, then there exist neural networks $\varphi_p, \varphi_\nu \in \bN^\sigma_{d, n_{\mathrm{in}}, 1}$ such that $\|p - \varphi_p\|_{\cC^2(\TT^d)}$ and $\|\nu - \varphi_\nu\|_{\cC^2(\TT^d)}$ are in $O\left(n_{\mathrm{in}}^{-2/(3d)}\right)$. In turn, this property leads to bounds on both the loss function of the algorithm and the error on the value function of the control problem.
The detailed analysis is left for future work.

\vskip 6pt
An important advantage of the DGM method is its flexibility and its generality since it can, a priori, be applied to almost any PDE. However, this generality can also be a drawback. Indeed the method does not exploit the structure of the PDE or in our case, of the PDE system under consideration. In generalizing this method to the case of our system, our main challenge was the choice of the relative weights to be assigned to the various terms in the loss function. These coefficients can be used to give more or less importance to some aspects of the solution. For instance, a large weight for the penalization terms ensures that the boundary and normalization conditions are likely to be satisfied at convergence. Also, if the gradient of one of the terms is too small, putting a larger weight can help the gradient descent to make faster progress. But the weights are hyperparameters that need to be tuned in accordance with other hyperparameters. Indeed, if they are not chosen appropriately, the stochastic gradient descent can easily be stuck in local minima. For instance if the weight of the normalization condition is not sufficiently large, the algorithm goes quickly towards a configuration which completely ignores this constraint and stays stuck there (although it should reduce the loss to try to satisfy this constraint). On the other hand, if the weight on the normalization condition is too large, it obfuscates the role of the other terms. Similarly, giving too much weight to one of the two PDEs prevents the neural networks from solving accurately the system.  
We had to find a good balance between the weights empirically. Furthermore, we found that it is sometimes helpful to adjust them dynamically during training to guide the neural network towards a satisfactory approximation of the solution. Although we are convinced that it can improve the results, we refrained from using this technique in the numerical examples presented below because of its \textit{ad hoc} nature.  In most machine learning applications, characterizing optimal combinations of hyperparameters is a very challenging task. In our case, finding optimal choices of weights and learning rates is an interesting question beyond the scope of the present work. Even without an optimal choice of parameters, we found that from a numerical standpoint, a good choice can be made by monitoring convergence to zero of each term of the loss function (the residuals and the penalty terms).

\section{Numerical Results}
\label{sec:num-res}

In this section we present numerical results obtained using implementations of the methods described in the previous sections. Algorithm 1 refers to the method based on minimization of the cost functional introduced in Section~\ref{sub:periodic_numerics}. Algorithm 2 refers to the DGM method described in Section~\ref{sec:DGM}. The implementations have been done in \texttt{TensorFlow}. The details of the neural network architectures are specified below. We used Adam for the gradient-based optimization. The results are presented on the unit torus (i.e., $[0,1]^d$ with periodic boundary conditions) instead of $\TT^d = [0,2\pi]^d$ as in the text.

\subsection{Examples in Dimension 1}

For ease of visualization, we start with univariate examples. We first consider models without explicit solutions, and we compare the solutions computed by the two algorithms introduced earlier with a benchmark solution computed by a deterministic method based on a finite difference scheme for the PDE system~\cite{MR2679575}. For these test cases, we used feed-forward neural networks with $3$ hidden layers, each layer having $20$ units.

\vskip 6pt
\noindent\textbf{Test case 1:} For the sake of illustration we include a model without mean field interaction, say
$$
	b(x, \mu, \phi) = \phi, 
	\qquad 
	f(x, \mu, \phi) = \frac{1}{2} |\phi|^2  +  \tilde f(x),
$$
with
\begin{equation}
\label{eq:ex-nonexpli-tilde-f-coscossin}
	\tilde f(x)  = 50 \Big(0.1\cos(2 \pi x) + \cos(4 \pi x) + 0.1 \sin\left(2\pi \left(x-\tfrac{\pi}{8}\right)\right)\Big),
\end{equation}
which has two local minima, one of them being a global minimum. The solution computed with the first algorithm is presented in Figure~\ref{fig:ex-nonexpli-p-nu}.

\begin{figure}
\centering
\begin{subfigure}{.45\textwidth}
  \centering
  \includegraphics[width=\linewidth]{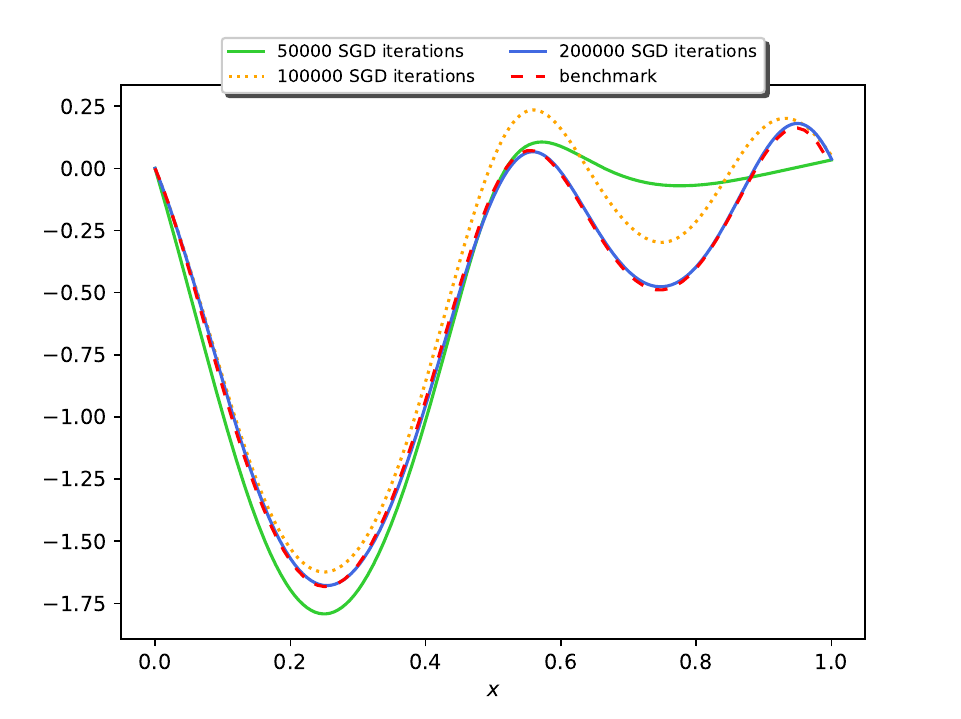}
  \caption*{Adjoint state $p$}
\end{subfigure}%
\begin{subfigure}{.45\textwidth}
  \centering
  \includegraphics[width=\linewidth]{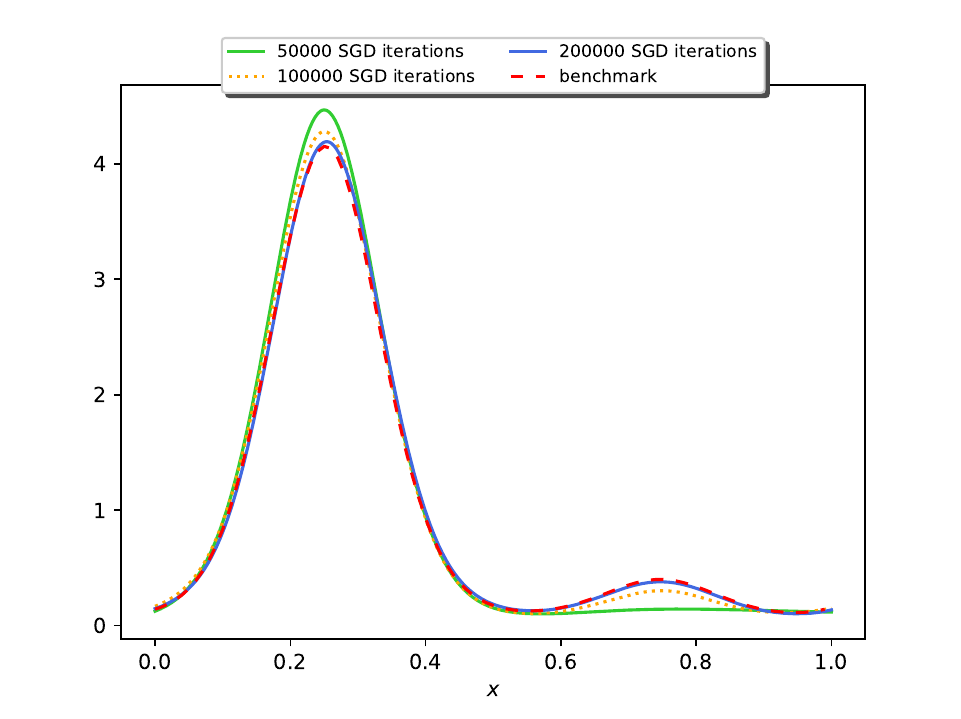}
  \caption*{Distribution $\nu$}
\end{subfigure}
\caption{Test case 1. Solution computed by Algorithm 1 and benchmark solution from deterministic method (dashed red line).}
\label{fig:ex-nonexpli-p-nu}
\end{figure}

\vskip 6pt
\noindent\textbf{Test case 2:} Next, we add a mean field interaction term in the cost
\begin{equation}
\label{eq:test2-b-f}
	b(x, \mu, \phi) = \phi, 
	\qquad 
	f(x, \mu, \phi) = \frac{1}{2} |\phi|^2  +  \tilde f(x) + |\mu(x)|^2,
\end{equation}
with $\tilde f$ given by~\eqref{eq:ex-nonexpli-tilde-f-coscossin}. Here $\mu(x)$ stands for the density of the measure $\mu$ at $x\in[0,1]$. The results are presented in Figure~\ref{fig:ex-nonexpliMF-p-nu}. Comparing with the first test case, one sees that due to the mean field term $|\mu(x)|^2$ in the cost function, the distribution is less concentrated around the global minimum and part of the mass is transferred to the second local minimum.

\begin{figure}
\centering
\begin{subfigure}{.45\textwidth}
  \centering
  \includegraphics[width=\linewidth]{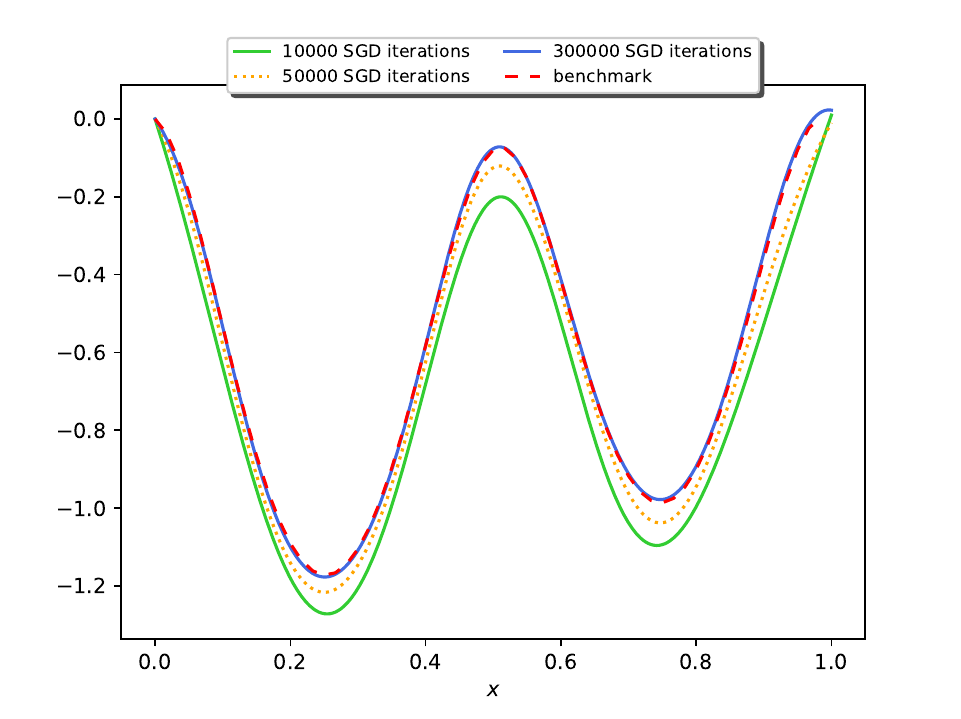}
  \caption*{Adjoint state $p$}
\end{subfigure}%
\begin{subfigure}{.45\textwidth}
  \centering
  \includegraphics[width=\linewidth]{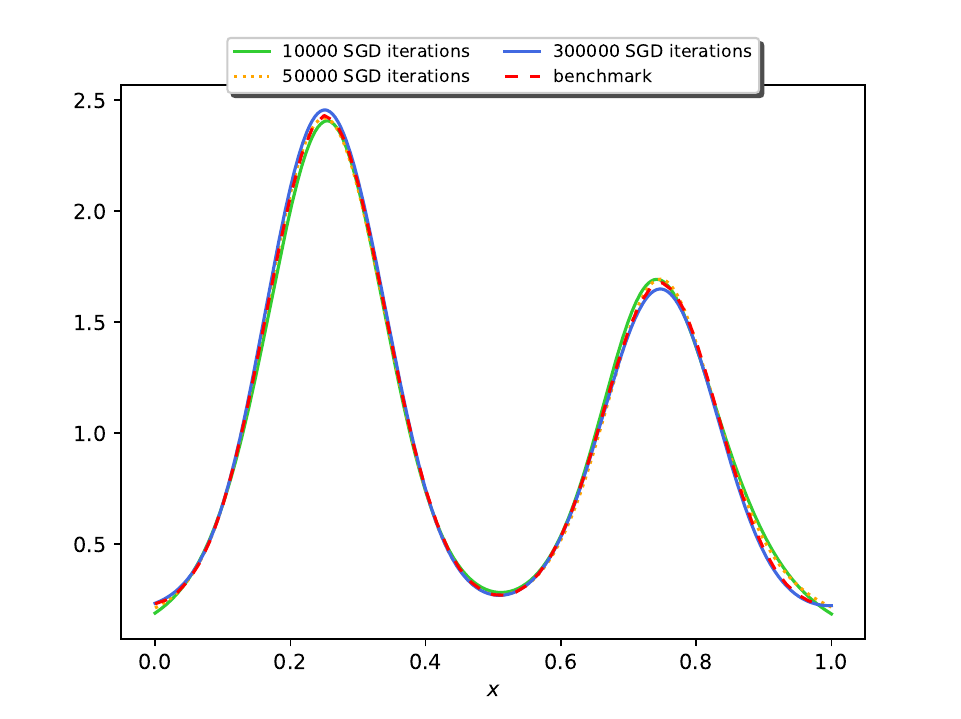}
  \caption*{Distribution $\nu$}
\end{subfigure}
\caption{Test case 2. Solution computed by Algorithm 1 and benchmark solution from deterministic method (dashed red line).}
\label{fig:ex-nonexpliMF-p-nu}
\end{figure}

\vskip 6pt
\noindent\textbf{Test case 3:} The DGM method can be used to solve the previous examples, but can also be used to solve other PDE systems, such as the one arising from mean field games. In the MFG setting, the PDE system for the optimality condition takes the following form~\cite{MR2269875}:
\begin{equation}
\label{fo:optimality-H-MFG}
\begin{cases}
&0=\frac12 \Delta \nu + \mathrm{div} \bigl( \partial_{y} H^\star(\cdot, \nu, \nabla p(\cdot)) \nu \bigr)\\
&0=\lambda + \frac12\Delta p(x) - H^\star(x, \nu, \nabla p(x)),
\end{cases}
\end{equation}
with normalization and boundary conditions, and where the minimized Hamiltonian $H^\star$ is defined in~\eqref{fo:H_star}. Notice that this PDE system \eqref{fo:optimality-H-MFG} is different from the PDE system \eqref{fo:optimality-H} for mean field control.

Taking $b$ and $f$ as in the previous test, namely~\eqref{eq:test2-b-f}, yields the solution displayed in Figure~\ref{fig:ex-nonexpliMF-p-nu-MFG} (to be compared with the corresponding curves for the MFC model of Test case 2, see Figure~\ref{fig:ex-nonexpliMF-p-nu}). Recall that with the DGM method, both $p$ and $\nu$ are approximated using two separate neural networks. In particular, we see on Figure~\ref{fig:ex-nonexpliMF-p-nu-MFG} that after $20000$ iterations of SGD, the neural network for $p$ has already roughly learnt the shape of the optimum, whereas the neural network for $\nu$ is still almost flat.

\begin{figure}
\centering
\begin{subfigure}{.45\textwidth}
  \centering
  \includegraphics[width=\linewidth]{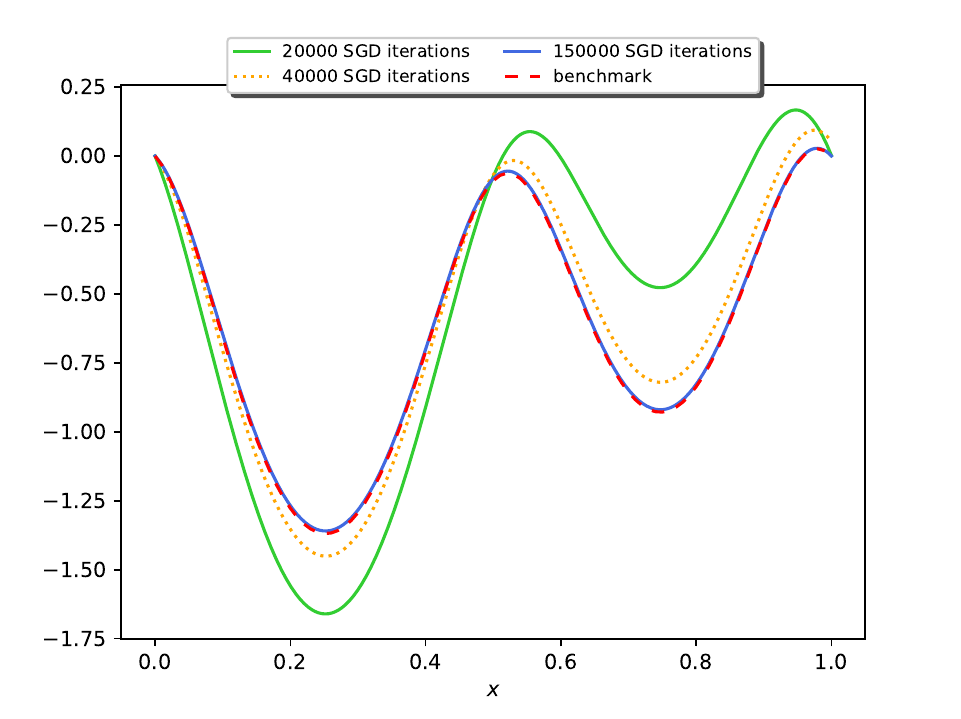}
  \caption*{Adjoint state $p$}
\end{subfigure}%
\begin{subfigure}{.45\textwidth}
  \centering
  \includegraphics[width=\linewidth]{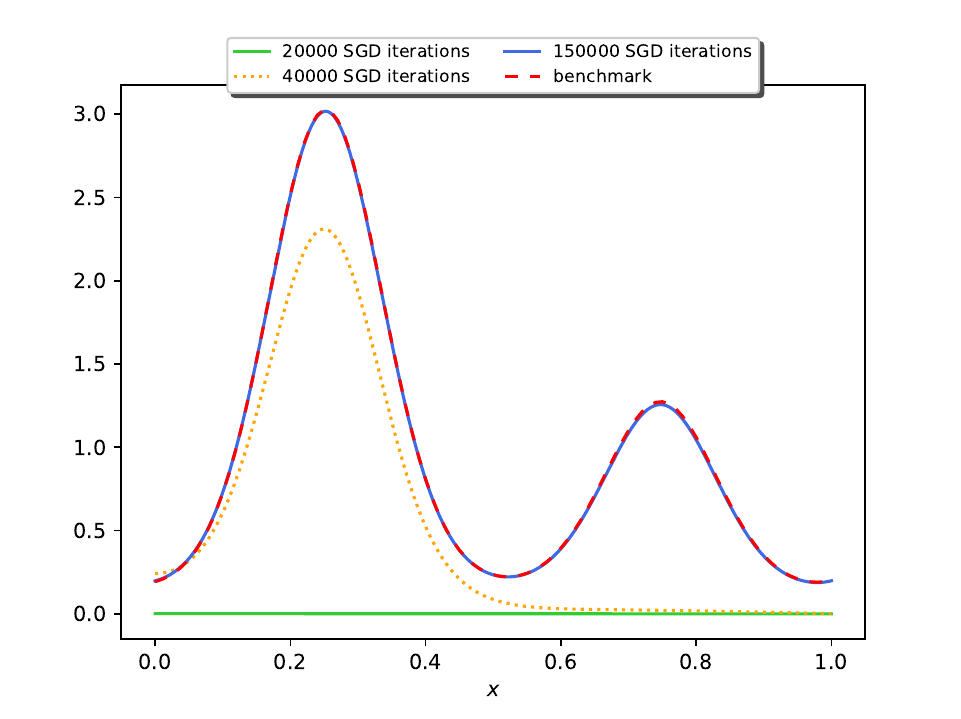}
  \caption*{Distribution $\nu$}
\end{subfigure}
\caption{Test case 3. Solution computed by Algorithm 2 (DGM) and benchmark solution from deterministic method (dashed red line).}
\label{fig:ex-nonexpliMF-p-nu-MFG}
\end{figure}

\subsection{Multivariate Examples with Explicit Solution}
\label{sub:ex-expli-sol-num}

To assess the quality of the proposed algorithms in higher dimension, we introduce simple toy models which can be solved explicitly. These models are very much in the spirit of examples considered in~\cite{MR3698446}. 
Let us take: 
\begin{equation}
\label{eq:testcase-bf}
	b(x,\alpha) = b(\alpha) = \alpha, \qquad f(x,\mu,\alpha) = \frac{1}{2} |\alpha|^2 + \tilde f(x) + \ln(\mu(x)),
\end{equation}
Then, the minimizer $\alpha^\star$ entering the definition of $H^\star(x,\mu,y)$  is given by:
$$
	\alpha^\star = \alpha^\star(y) = y.
$$
So $H^\star(x,\mu,y) = - \frac{1}{2}|y|^2 + \tilde f(x) + \ln(\mu(x))$ and $\int D_\mu H^\star(\xi,\nu, \nabla p(\xi))(x)\, \nu(d\xi) = 1$.
The PDE system~\eqref{fo:optimality-H} rewrites:
\begin{equation}
\label{fo:PDEsystem-simple}
\begin{cases}
&0= \frac12 \Delta \nu(x) - \mathrm{div} \bigl( \nabla p \nu \bigr)(x)
\\
&0= \lambda + \frac12 \Delta p(x) + \frac{1}{2}|\nabla p(x)|^2 - \tilde f(x) - \left( \ln(\nu(x)) + 1\right).
\end{cases}
\end{equation}
Assuming the existence of a smooth enough solution $(\nu,p,\lambda)$, the first equation allows us to express $\nu$ in terms of $p$ as follows: 
\begin{equation}
	\label{eq:nu-expp}
	\nu(x) = \frac{e^{2p(x)}}{\int e^{2p(x')}dx'}.
\end{equation}
The second equation in~\eqref{fo:PDEsystem-simple} rewrites
$$
	\nu(x) = \frac{e^{\frac{1}{2} (\Delta p(x) + |\nabla p(x)|^2) - \tilde f(x)}}{e^{1 - \lambda}}.
$$
In this case, the PDE system~\eqref{fo:PDEsystem-simple} is solved provided the above equation and the second equation in~\eqref{fo:PDEsystem-simple} are satisfied, which means that $(p,\lambda)$ solves:
$$
	2p(x) = \frac{1}{2} (\Delta p(x) + |\nabla p(x)|^2) - \tilde f(x), \qquad \lambda = 1 - \ln(\int e^{2p}).
$$
We consider two specific instances of $\tilde f$ for which we are able to obtain closed-form expressions for $p$.

\vskip 6pt
\noindent\textbf{Test case 4:} We consider~\eqref{eq:testcase-bf} again but now with $\tilde f$ given by
$$
	\tilde f(x) = 2 \pi^2 \left[ - \sum_{i=1}^d \sin(2 \pi x_i) + \sum_{i=1}^d |\cos(2 \pi x_i)|^2 \right] - 2\sum_{i=1}^d \sin(2 \pi x_i),
$$
then the solution is given by $p(x) = \sum_{i=1}^d \sin(2 \pi x_i)$ and $\lambda = 1 - \ln(\int e^{2 \sum_{i=1}^d \sin(2 \pi \xi_i)} d \xi)$.

We have solved numerically this problem in dimension $d=4$ using both methods.  
The convergence of the approximation $p_\theta$ learnt by our first algorithm towards the analytical solution $p$ is presented in Figure~\ref{fig:ex-expliMF-p-CV2expli}. This figure shows the relative $L^2$-error, which is defined as
$$
	\left(\frac{\|p - p_\theta\|^2_2}{\|p\|^2_2}\right)^{1/2} = \sqrt{\left( \int_{\TT^d} |p(x) - p_\theta(x)|^2 dx \right) / \left( \int_{\TT^d} |p(x)|^2 dx\right)}.
$$
In the implementation, this quantity is estimated with $10^5$ Monte Carlo samples for each integral. The figure corresponds to one run of SGD and illustrates the fact that the algorithm can be stuck in a local minimum for a certain number of iterations (between roughly iterations $10^5$ and $2 \times 10^5$ on this example) before finding its way out to a better solution. The distribution $\nu$ is deduced from $p$ using the formula~\eqref{eq:nu-expp}, which explains why the two convergence curves have the same shape.

\begin{figure}
\centering
\begin{subfigure}{.45\textwidth}
  \centering
  \includegraphics[width=\linewidth]{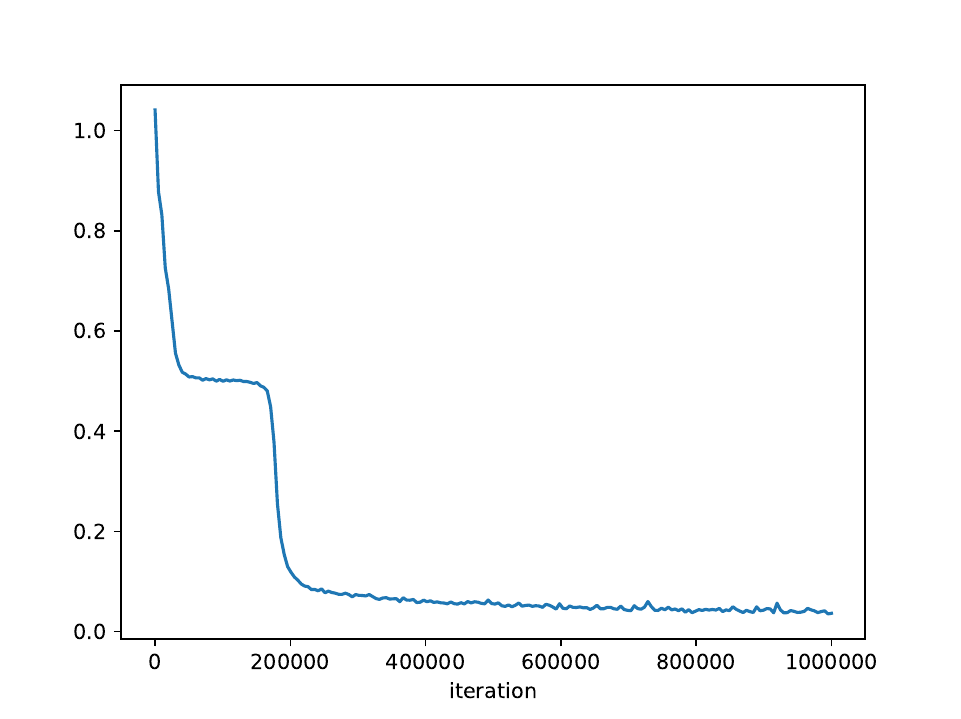}
  \caption*{Relative $L^2-$error on adjoint state $p$}
\end{subfigure}%
\begin{subfigure}{.45\textwidth}
  \centering
  \includegraphics[width=\linewidth]{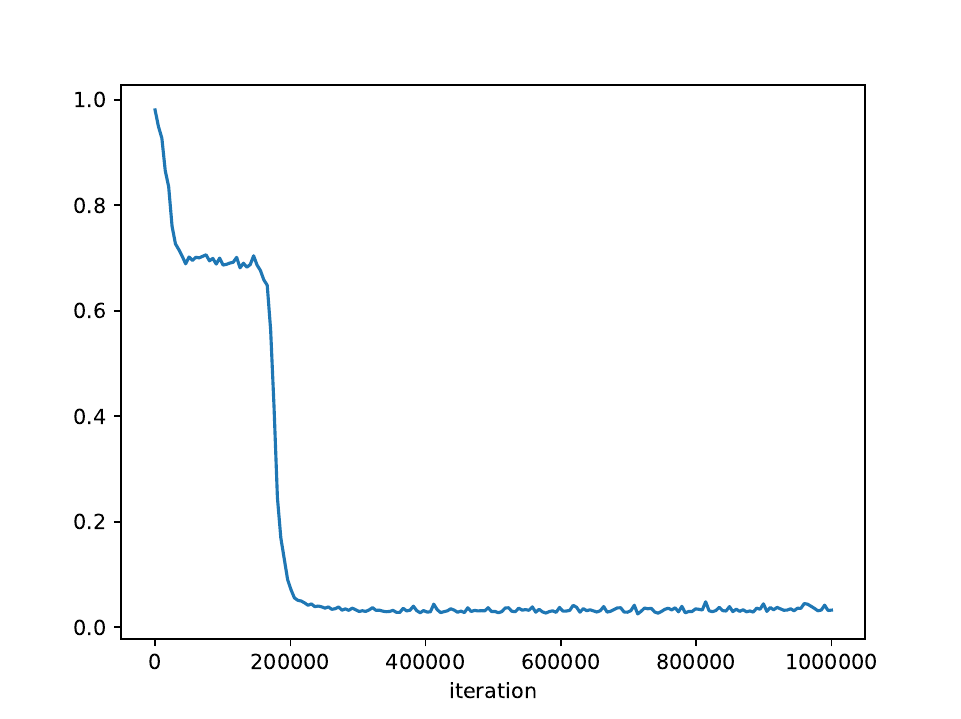}
  \caption*{Relative $L^2-$error on adjoint state $\nu$}
\end{subfigure}
\caption{Test case 4. Relative $L^2$-error by Algorithm 1.}
\label{fig:ex-expliMF-p-CV2expli}
\end{figure}

For the DGM method, numerical convergence is presented in Figure~\ref{fig:ex-expliMF-p-CV2expli-DGM}. As in the previous test case, the convergence rate of $p$ and $\nu$ is quite different because they are approximated by distinct neural networks. In particular, the error on $\nu$ decreases at a lower rate and suffers from a larger noise, which could be due to the form of the solution, see~\eqref{eq:nu-expp}. Moreover, since the logarithm of $\nu$ appears in the HJB equation, we were forced to choose a positive-valued function for the activation function of the output layer (instead of the identity). We compared results obtained with the exponential function and the softplus function: $\psi(x) = \ln(1+e^x)$. Since they gave similar results, we provide here the ones obtained with the exponential function.  

Here, for both methods, we used feed-forward neural networks with $3$ hidden layers, each layer having $20$ units.

\begin{figure}
\centering
\begin{subfigure}{.45\textwidth}
  \centering
  \includegraphics[width=\linewidth]{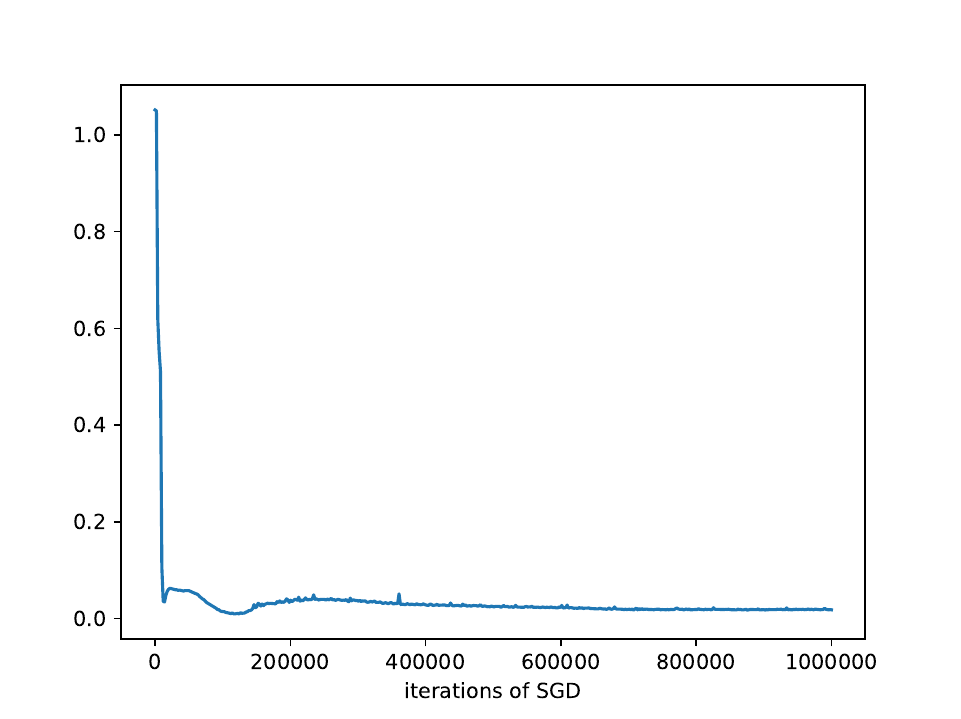}
  \caption*{Relative $L^2-$error on adjoint state $p$}
\end{subfigure}%
\begin{subfigure}{.45\textwidth}
  \centering
  \includegraphics[width=\linewidth]{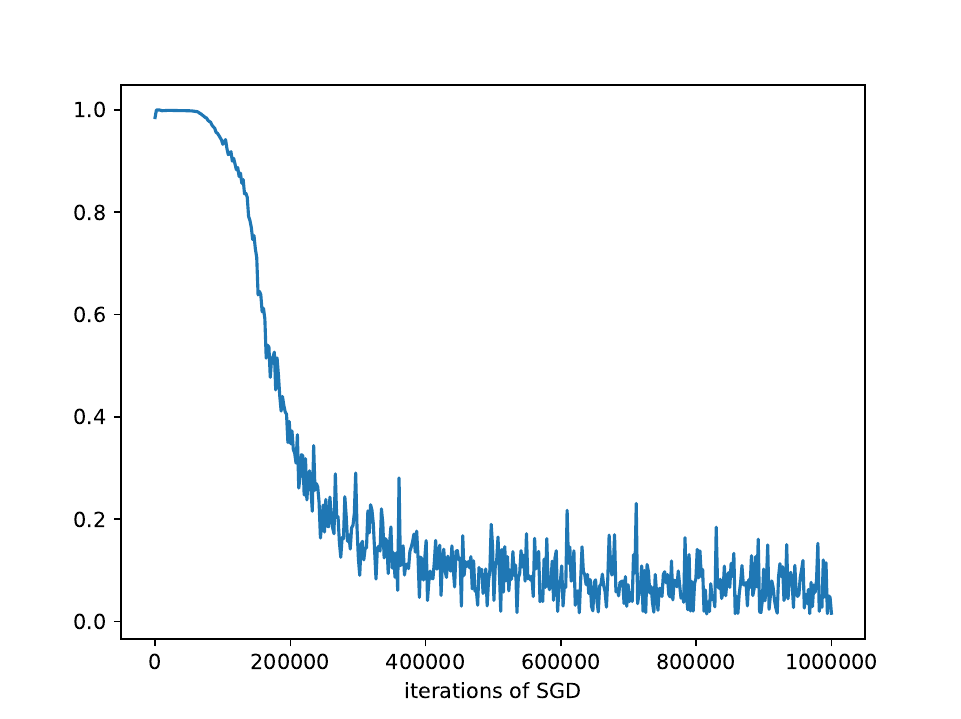}
  \caption*{Relative $L^2-$error on distribution $\nu$}
\end{subfigure}
\caption{Test case 4. Relative $L^2-$error by Algorithm 2 (DGM).}
\label{fig:ex-expliMF-p-CV2expli-DGM}
\end{figure}

\vskip 6pt
\noindent\textbf{Test case 5:} We consider a variant of the previous test case where  $\tilde f$ is chosen such that
$$
	p(x) = \displaystyle \prod_{i=1}^d \sin(2 \pi x_i)^2 - \int_{\TT^d} \prod_{i=1}^d \sin(2 \pi x_i)^2 dx.
$$
We use this example to study the influence of the number of hidden units and the number of samples in the population on the approximation $p_\theta$ found by the algorithm. For simplicity and to be consistent with the theoretical bounds provided in the prequel, we consider here neural networks with a single hidden layer. Figure~\ref{fig:ex-explitMF2-p-CMPunits} illustrates the dependence on the number of hidden units. As seen on Figure~\ref{fig:ex-explitMF2-p-CMPunits-fct-of-nU}, the error decreases quickly as the number of units grows until $30$. However, for a number of units larger than $30$, the error almost stagnates. This is due to the fact that the number of samples in the population drawn at each iteration of SGD is kept fixed to $10^5$. The dependence on this number of samples is illustrated in Figure~\ref{fig:ex-explitMF2-p-CMPsamples}, while keeping the number of units fixed to $60$. These numerical results were obtained by averaging over $10$ runs of SGD.

\begin{figure}
\renewcommand{\thesubfigure}{\Alph{subfigure}}
\centering
\begin{subfigure}{.45\textwidth}
  \centering
  \includegraphics[width=\linewidth]{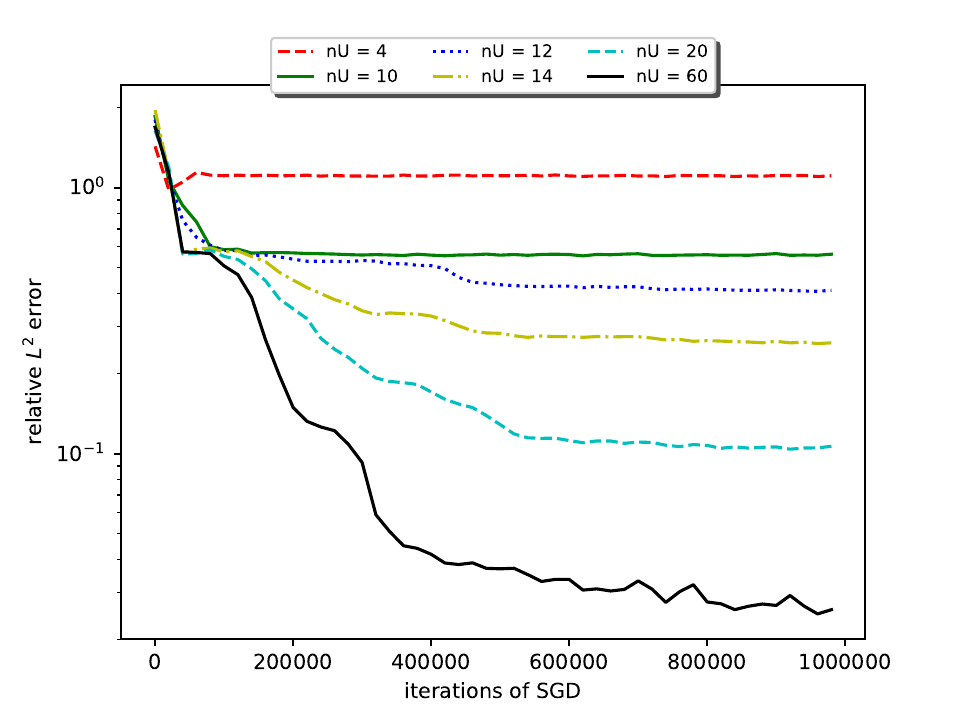}
  \caption{Relative error vs number of iterations of SGD}
  \label{fig:ex-explitMF2-p-CMPunits-fct-of-iter}
\end{subfigure}%
\begin{subfigure}{.45\textwidth}
  \centering
  \includegraphics[width=\linewidth]{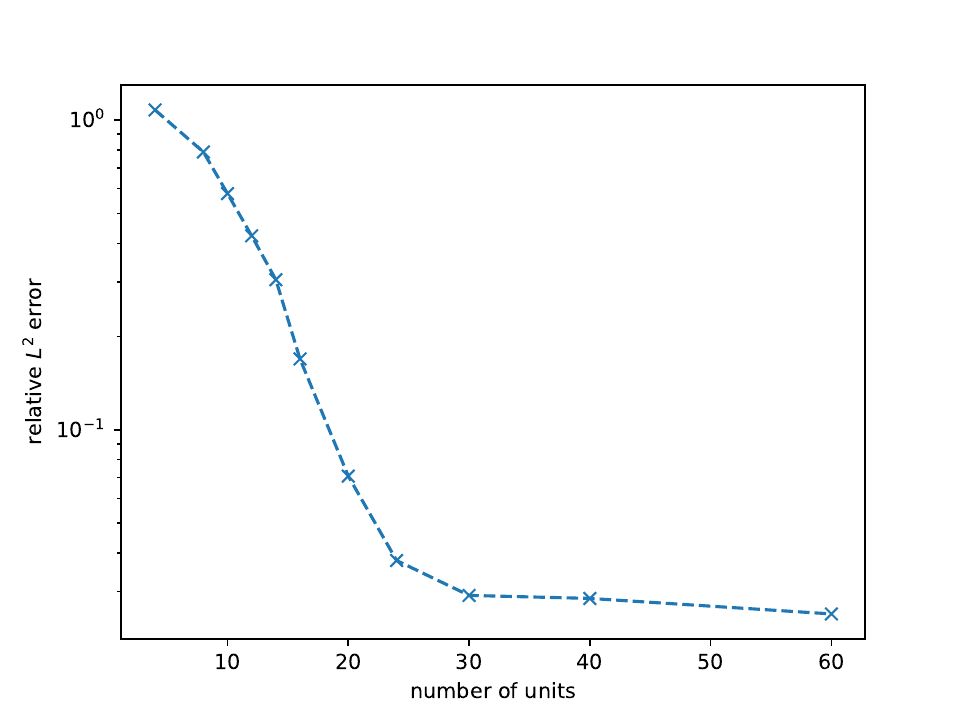}
  \caption{Relative error vs number of units}
  \label{fig:ex-explitMF2-p-CMPunits-fct-of-nU}
\end{subfigure}

\caption{Test case 5. Dependence of the relative $L^2$ error on the number of hidden units (Algorithm 1). The number of samples in the population is $10^5$. The number of units is indicated by ``nU'' for the figure on the left. On the right, relative $L^2$ error after $10^6$ iterations.}
\label{fig:ex-explitMF2-p-CMPunits}
\end{figure}

\begin{figure}
\centering
\begin{subfigure}{.45\textwidth}
  \centering
  \includegraphics[width=\linewidth]{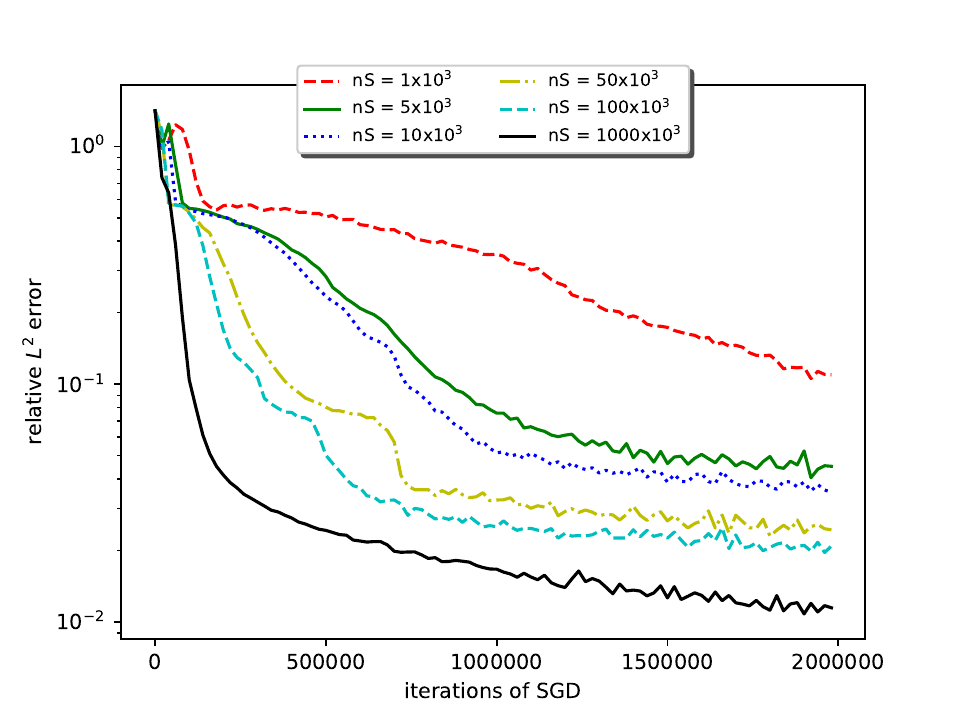}
  \caption*{Relative error vs number of iterations of SGD}
\end{subfigure}%
\begin{subfigure}{.45\textwidth}
  \centering
  \includegraphics[width=\linewidth]{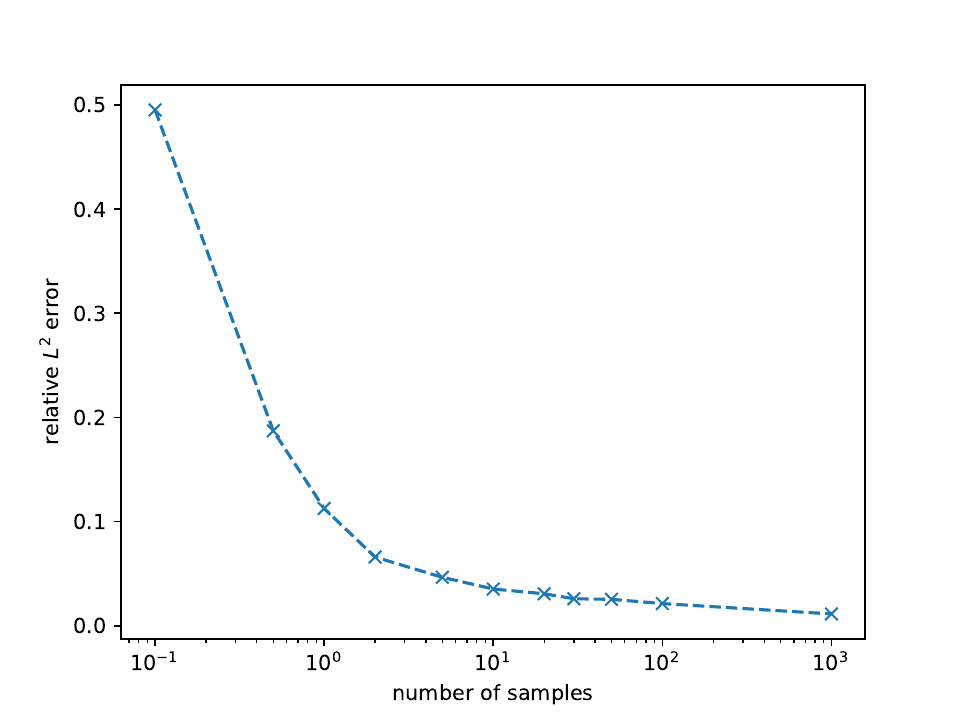}
  \caption*{Relative error vs number of samples}
\end{subfigure}
\caption{Test case 5. Dependence of the relative $L^2$ error on the number of samples in a population (Algorithm 1). The number of hidden units it $60$ for all curves. The number of samples is indicated by ``nS'' for the figure on the left. On the right, we plot the relative $L^2$ error after $10^6$ iterations.}
\label{fig:ex-explitMF2-p-CMPsamples}
\end{figure}

\vskip 6pt
\noindent\textbf{Test case 6:} We consider the following variant of the PDE system~\eqref{fo:PDEsystem-simple}, which corresponds to a MFG with the same cost function and dynamics as the MFC problem described above:
\begin{equation}
\label{fo:PDEsystem-simple-MFG}
\begin{cases}
&0= \frac12 \Delta \nu(x) - \mathrm{div} \bigl( \nabla p \nu \bigr)(x)
\\
&0= \lambda + \frac12 \Delta p(x) + \frac{1}{2}|\nabla p(x)|^2 - \tilde f(x) - \ln(\nu(x)).
\end{cases}
\end{equation}
This system does not have a variational interpretation so we approach it using Algorithm 2. Although neural networks with simple structures are universal approximators~\cite{cybenko1989approximation,hornik1991approximation}, in practice using deep neural networks with a well-suited architecture is crucial to the success of many deep learning applications. Here, we tested two different architectures. The first one is a simple feedforward fully connected architecture, as presented above. The second one is a recurrent neural network architecture, and more specifically the one -- referred to as DGM architecture in the sequel -- proposed by Sirignano and Spiliopoulos in~\cite{MR3874585} which is inspired by long short-term memory networks~\cite{hochreiter1997long} and highway networks~\cite{srivastava2015training}. It has been shown empirically to be particularly well adapted to learn how to approximate functions with complex structures as well as derivatives of functions. For the sake of brevity we do not reproduce here the architecture and we refer to~\cite{MR3874585} for more details about this type of neural networks and its success on some classes of PDEs. In high dimension, we did not obtain good results using the fully connected architecture so we present only numerical results obtained using the second type of architecture.

Figure~\ref{fig:ACD-d10} shows the $L^2$ errors on $p$ and $\nu$ and the residuals for both PDEs in dimension $d=10$. Here, we used a three layer neural network with $100$ units per layer for $p$ and three layer neural network with $200$ units per layer for $\nu$. As for the learning rate, instead of plain Adam optimizer as in the previous case, we used Adam with a piecewise-constant schedule of learning rate decay defined by $\alpha_m = 1\times 10^{-3} \mathbbm{1}_{1 \le m < 20000 } + 3\times 10^{-4} \mathbbm{1}_{20000 \le m < 40000 } + 1\times 10^{-4} \mathbbm{1}_{40000 \le m < 60000 } + 3\times 10^{-5} \mathbbm{1}_{60000 \le m}$. We used a minibatch of $4096$ samples at each iteration. The curves in Figure~\ref{fig:ACD-d10} have been obtained with a single run of the algorithm, by saving the values every $100$ iterations and taking a moving average over $10$ points to make the curves more readable. 

\begin{figure}
\centering
\begin{subfigure}{.45\textwidth}
  \centering
  \includegraphics[width=\linewidth]{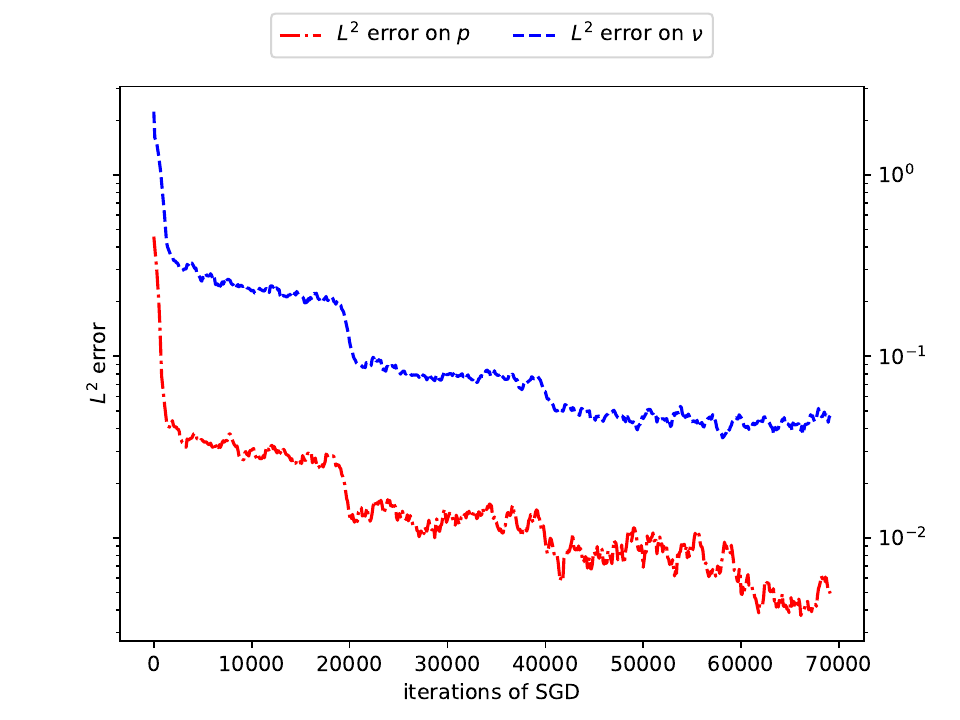}
  \caption*{$L^2$ errors on $p$ and $\nu$}
\end{subfigure}%
\begin{subfigure}{.45\textwidth}
  \centering
  \includegraphics[width=\linewidth]{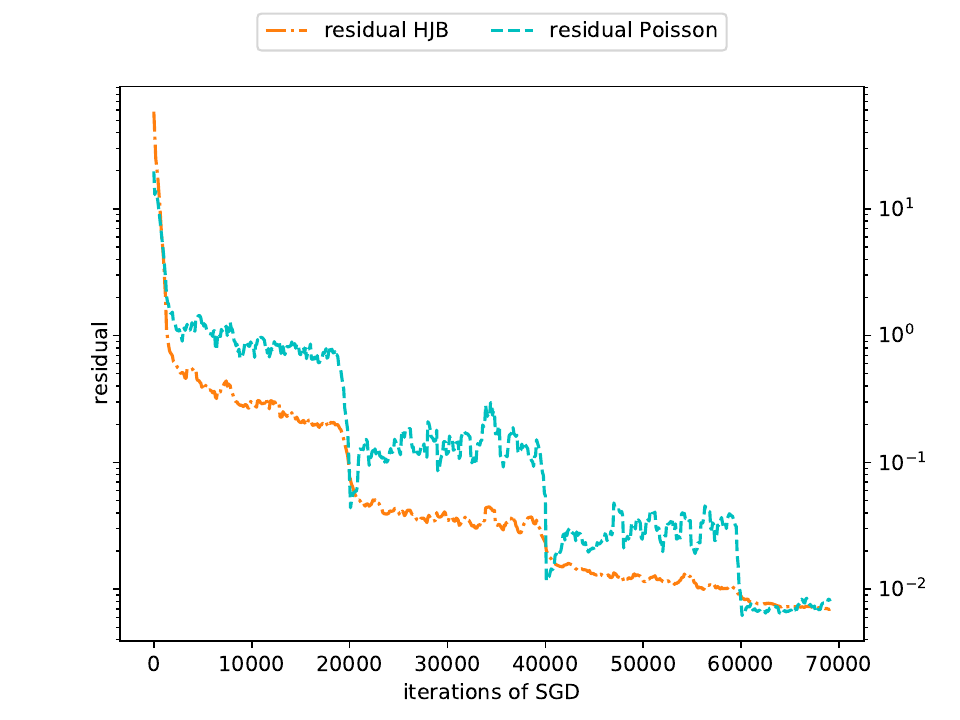}
  \caption*{Residuals for HJB and Poisson equations}
\end{subfigure}
\caption{Test case 6. $L^2$ error for each function (left) and residuals for each PDE versus number of iterations using Algorithm 2 (right).} 
\label{fig:ACD-d10}
\end{figure}

\vskip 6pt
\noindent\textbf{Test case 7:} We consider the following example with a quadratic dependence on the distribution, inspired by~\cite[Section 6]{MR2888257} for which there is no analytical solution to the best of our knowledge. Here we take 
$$
	H^\star(x,\mu,y) = - \frac{1}{2}|y|^2 + \tilde f(x) + |\mu(x)|^2
$$
with
\begin{equation}
\label{eq:tildef-sinpcos}
	\tilde f(x) = C \, \frac{1}{d} \, \sum_{i=1}^d \left[ \sin(2 \pi x_i) + \cos(2 \pi x_i) \right],
\end{equation}
whose maximum and minimum values are $\pm \sqrt{2} C$. 
We solved this example in dimensions up to $d=100$ using Algorithm 2 and the DGM architecture, with $c=1.5$. Since the PDEs no longer involve the logarithm of $\nu$, we use the identity function for the output layer's activation function. The residuals of the two PDEs are shown in Figure~\ref{fig:sinpcos-mu2}, which was obtained by averaging over three runs of the algorithm for each dimension, saving the residuals every 100 iterations, and using a moving average over 10 points to make the curves more readable.  We used minibatches of $128$ samples, and a DGM architecture with $3$ layers and $200$ units per layer. The learning rate was updated using Adam optimizer initialized with the value $10^{-4}$. Although the residuals are much smaller in dimension $d=2$ than in $d=50$, there is no deterioration between $d=50$ and $d=100$.   These results show that the method provides a good approximate solution without intensive tuning of the hyperparameters. Tuning these hyperparameters (e.g., the architecture, the learning rate and the minibatch size) based on the investigation carried out in the previous test cases should lead to even better results. But to be able to easily compare the results in various dimensions, we decided to keep them fixed. Only for the sake of comparison between dimensions, the average time per iteration is indicated in Table~\ref{table:test7-8-time}. These results have been obtained using a GPU with 2.4 GHz Xeon Broadwell
E5-2680 v4 processor.   Recall that here the architecture and the minibatch size are fixed, so the increase in computational time is due to the optimization of the neural network and these results show that for this aspect the computational cost scales almost linearly with the dimension.

\begin{table}[!ht]
\caption{
	Average time per iteration in seconds for test cases 7 and 8, with $d \in \{2,10,50,100\}$.
	\label{table:test7-8-time}}
\small
\centering
\ra{1.1}
\begin{tabular}{@{}cccc@{}}
\toprule
\textbf{Dimension $d$}&
Test case 6 & Test case 7 
\\
\midrule
$2$ & $5 \times 10^{-2}$ s. & $6 \times 10^{-2}$ s.
\\ 
$10$ & $21 \times 10^{-2}$ s. & $22 \times 10^{-2}$ s. 
\\ 
$50$ & $109 \times 10^{-2}$ s. & $108 \times 10^{-2}$ s.  
\\ 
$100$ & $225 \times 10^{-2}$ s. & $229 \times 10^{-2}$ s. 
\\ 
\bottomrule
\end{tabular}
\end{table}

\begin{figure}
\centering
\begin{subfigure}{.45\textwidth}
  \centering
  \includegraphics[width=\linewidth]{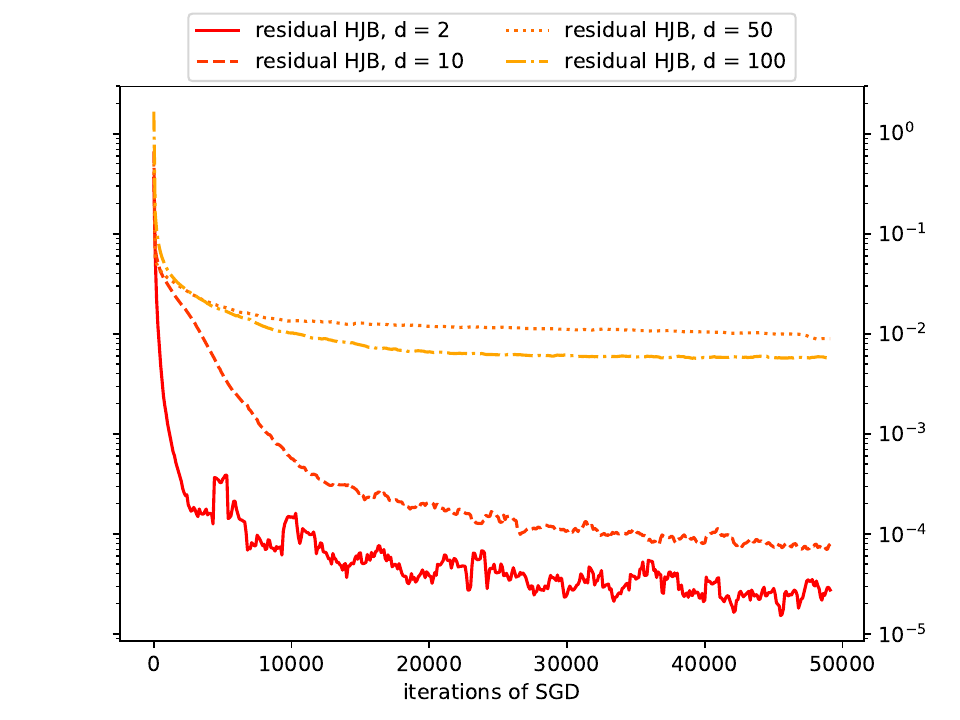}
  \caption*{Residuals for HJB equation}
\end{subfigure}%
\begin{subfigure}{.45\textwidth}
  \centering
  \includegraphics[width=\linewidth]{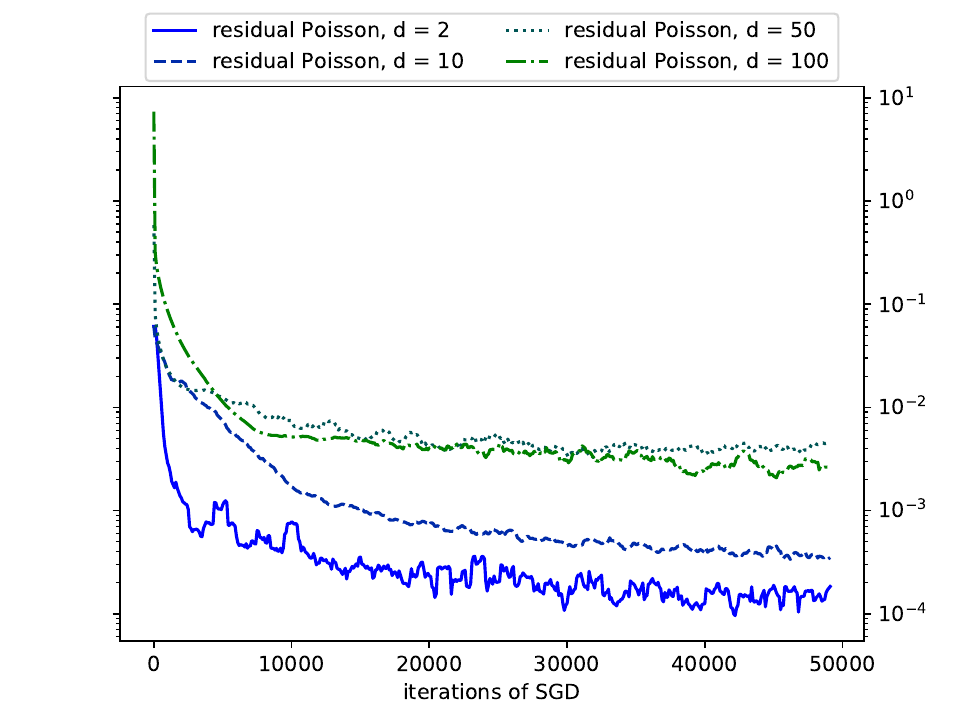}
  \caption*{Residuals for Poisson equation}
\end{subfigure}
\caption{Test case 7. Residuals versus number of iterations using Algorithm 2, for $d=2, 10, 50, 100$.} 
\label{fig:sinpcos-mu2}
\end{figure}

\vskip 6pt
\noindent\textbf{Test case 8:} We consider the following variant which incorporates a non-local term:
\begin{equation}
\label{eq:test7-H}
	H^\star(x,\mu,y) = - \frac{1}{2}|y|^2 + \tilde f(x) + |\mu(x)|^2 + C \, \frac{1}{d} \, \sum_{i=1}^d \left[ \int_{\TT^d} \sin(x') d \mu(x') \right]_i
\end{equation}
with $\tilde f$ given by~\eqref{eq:tildef-sinpcos}.  
We solved this example in dimensions up to $d=100$ using the DGM method, with $c=1.5$. The residuals of the two PDEs are shown in Figure~\ref{fig:sinpcos-mu2-integ}, which was obtained using the same architecture and the same procedure as for Test case 6.  The average time per iteration is indicated in Table~\ref{table:test7-8-time} and is roughly the same as in Test case 6 because we used $128$ samples too to approximate the integral appearing in~\eqref{eq:test7-H}.

\begin{figure}
\centering
\begin{subfigure}{.45\textwidth}
  \centering
  \includegraphics[width=\linewidth]{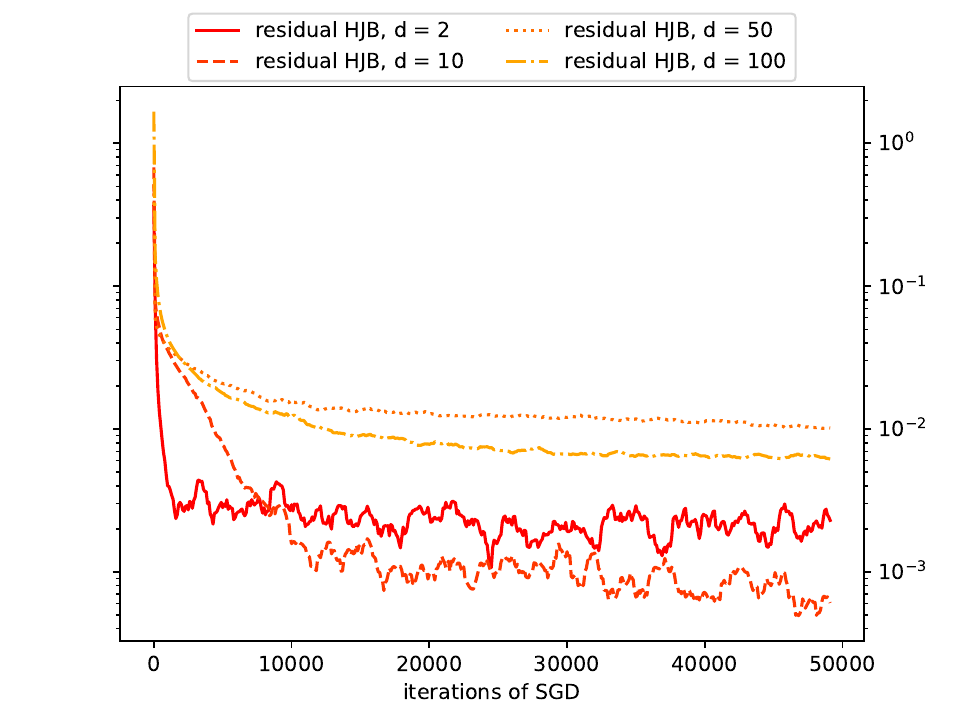}
  \caption*{Residuals for HJB equation}
\end{subfigure}%
\begin{subfigure}{.45\textwidth}
  \centering
  \includegraphics[width=\linewidth]{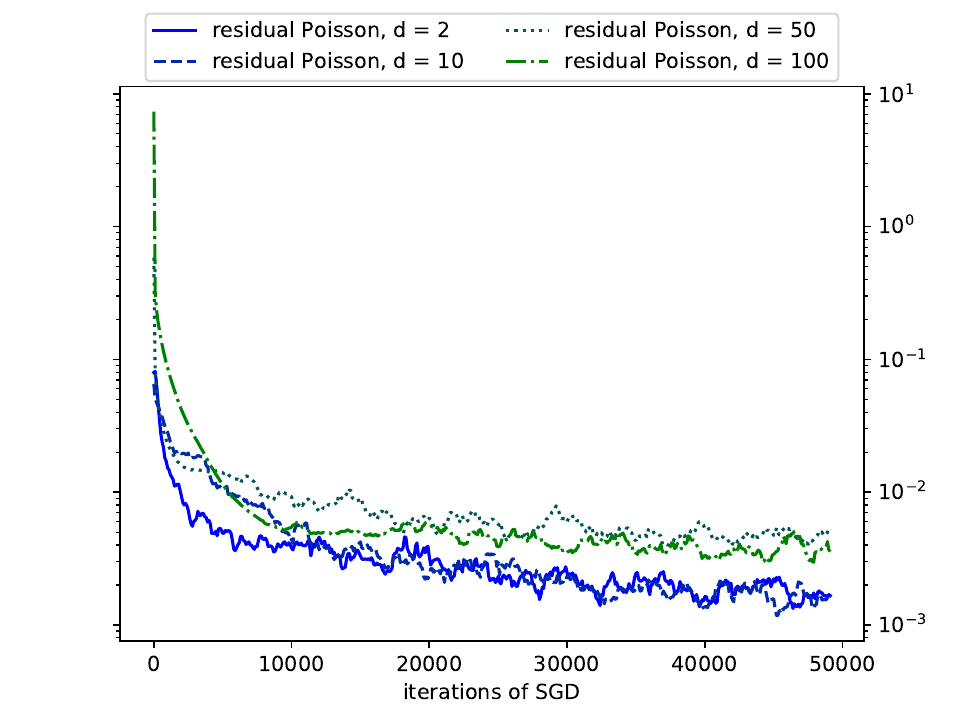}
  \caption*{Residuals for Poisson equation}
\end{subfigure}
\caption{Test case 8. Residuals versus number of iterations (Algorithm 2), for $d=2, 10, 50, 100$.} 
\label{fig:sinpcos-mu2-integ}
\end{figure}

\section{Conclusion}
In this paper, we introduced two numerical algorithms for the solution of the optimal control of  ergodic McKean-Vlasov dynamics
also known as ergodic mean field control problems. We approximated the theoretical solutions by functions given by neural networks, the latter being determined by their architectures and suitable sets of parameters. This allowed the use of modern machine learning tools, and efficient implementations of stochastic gradient descent.

The first algorithm is based on the specific structure of  the ergodic optimal control problem.
We provided a mathematical proof of the convergence of the algorithm, and we analyzed rigorously the numerical scheme by controlling both the approximation and the estimation error. The second method is an adaptation of the deep Galerkin method to the system of partial differential equations issued from the optimality conditions. We showed that it can also be applied to the PDE system arising in mean field games, even when the latter do not have a variational structure. 

Our numerical results support the idea that these methods can be used in large dimension. From here, several directions can be contemplated. First, using the same algorithms and architectures, it should be possible to obtain even better results with more time and better hardware (i.e. computational time and power). We view this work as a first step, and we tried to frame it so that it could be accessible to, and reproducible by a large community of researchers. 

Finally, we believe that it should be possible to design efficient methods by using known properties of the specific structure of the mean-field PDE system without assuming any knowledge of the form of the solution. The adjoint structure between the HJB and the Poisson equation is a case in point.

\bibliographystyle{abbrv}%

\begin{thebibliography}{}

\end{thebibliography}


\small
\begin{thebibliography}{10}

\bibitem{MR2888257}
Y.~Achdou, F.~Camilli, and I.~Capuzzo-Dolcetta.
\newblock Mean field games: numerical methods for the planning problem.
\newblock {\em SIAM J. Control Optim.}, 50(1):77--109, 2012.

\bibitem{MR2679575}
Y.~Achdou and I.~Capuzzo-Dolcetta.
\newblock Mean field games: numerical methods.
\newblock {\em SIAM J. Numer. Anal.}, 48(3):1136--1162, 2010.

\bibitem{MR3392611}
Y.~Achdou and M.~Lauri\`ere.
\newblock On the system of partial differential equations arising in mean field
  type control.
\newblock {\em Discrete Contin. Dyn. Syst.}, 35(9):3879--3900, 2015.

\bibitem{MR3575615}
Y.~Achdou and M.~Lauri{\`e}re.
\newblock Mean {F}ield {T}ype {C}ontrol with {C}ongestion ({II}): {A}n
  augmented {L}agrangian method.
\newblock {\em Appl. Math. Optim.}, 74(3):535--578, 2016.

\bibitem{MR3698446}
N.~Almulla, R.~Ferreira, and D.~Gomes.
\newblock Two numerical approaches to stationary mean-field games.
\newblock {\em Dyn. Games Appl.}, 7(4):657--682, 2017.

\bibitem{balata2019class}
A.~Balata, C.~Hur{\'e}, M.~Lauri{\`e}re, H.~Pham, and I.~Pimentel.
\newblock A class of finite-dimensional numerically solvable mckean-vlasov
  control problems.
\newblock {\em ESAIM: Proceedings and Surveys}, 65:114--144, 2019.

\bibitem{MR3644590}
J.-D. Benamou, G.~Carlier, and F.~Santambrogio.
\newblock Variational mean field games.
\newblock In {\em Active particles. {V}ol. 1. {A}dvances in theory, models, and
  applications}, Model. Simul. Sci. Eng. Technol., pages 141--171.
  Birkh\"{a}user/Springer, Cham, 2017.

\bibitem{MR3134900}
A.~Bensoussan, J.~Frehse, and S.~C.~P. Yam.
\newblock {\em Mean field games and mean field type control theory}.
\newblock Springer Briefs in Mathematics. Springer, New York, 2013.

\bibitem{MR1370849}
M.~Bossy and D.~Talay.
\newblock A stochastic particle method for the {M}c{K}ean-{V}lasov and the
  {B}urgers equation.
\newblock {\em Math. Comp.}, 66(217):157--192, 1997.

\bibitem{BricenoAriasetalCEMRACS2017}
L.~M. Brice\~no Arias, D.~Kalise, Z.~Kobeissi, M.~Lauri\`ere,
  A.~Mateos~Gonz\'alez, and F.~J. Silva.
\newblock On the implementation of a primal-dual algorithm for second order
  time-dependent mean field games with local couplings.
\newblock {\em ESAIM: ProcS}, 65:330--348, 2019.

\bibitem{MR3772008}
L.~M. Brice\~{n}o Arias, D.~Kalise, and F.~J. Silva.
\newblock Proximal methods for stationary mean field games with local
  couplings.
\newblock {\em SIAM J. Control Optim.}, 56(2):801--836, 2018.

\bibitem{Butkovsky}
O.~Butkovsky.
\newblock On ergodic properties of nonlinear {M}arkov chains and stochastic
  {M}c{K}ean - {V}lasov equations.
\newblock {\em Theory of Probability and Applications, Society for Industrial
  and Applied Mathematics}, 58(4):661 -- 674, 2014.

\bibitem{MR3882530}
S.~Cacace, F.~Camilli, A.~Cesaroni, and C.~Marchi.
\newblock An ergodic problem for mean field games: qualitative properties and
  numerical simulations.
\newblock {\em Minimax Theory Appl.}, 3(2):211--226, 2018.

\bibitem{Cardaliaguet-2013-notes}
P.~Cardaliaguet.
\newblock Notes on mean field games.
\newblock 2013.

\bibitem{MR3148086}
E.~Carlini and F.~J. Silva.
\newblock A fully discrete semi-{L}agrangian scheme for a first order mean
  field game problem.
\newblock {\em SIAM J. Numer. Anal.}, 52(1):45--67, 2014.

\bibitem{CarmonaDelarue_book_I}
R.~Carmona and F.~Delarue.
\newblock {\em Probabilistic Theory of Mean Field Games {I}: Mean Field FBSDEs,
  Control, and Games}.
\newblock Stochastic Analysis and Applications. Springer Verlag, 2017.

\bibitem{CarmonaDelarue_book_II}
R.~Carmona and F.~Delarue.
\newblock {\em Probabilistic Theory of Mean Field Games {II}: Mean Field Games
  with Common Noise and Master Equations}.
\newblock Stochastic Analysis and Applications. Springer Verlag, 2017.

\bibitem{CarmonaLauriere_DL}
R.~Carmona and M.~Lauri{\`e}re.
\newblock Convergence analysis of machine learning algorithms for the numerical
  solution of mean field control and games: {II} - {T}he finite horizon case.
\newblock Preprint, arXiv:1908.01613.

\bibitem{CarmonaLauriereTan-2019-LQMFRL}
R.~Carmona, M.~Lauri{\`e}re, and Z.~Tan.
\newblock Linear-quadratic mean-field reinforcement learning: Convergence of
  policy gradient methods.
\newblock Preprint, arXiv:1910.04295.

\bibitem{MR3914553}
J.-F. Chassagneux, D.~Crisan, and F.~Delarue.
\newblock Numerical method for {FBSDE}s of {M}c{K}ean-{V}lasov type.
\newblock {\em Ann. Appl. Probab.}, 29(3):1640--1684, 2019.

\bibitem{Chow_et_al}
Y.~Chow, W.~Li, S.~Osher, and W.~Yin.
\newblock Algorithm for hamilton-jacobi equations in density space via a
  generalized hopf formula.
\newblock Technical report, 2018.

\bibitem{cybenko1989approximation}
G.~Cybenko.
\newblock Approximation by superpositions of a sigmoidal function.
\newblock {\em Mathematics of control, signals and systems}, 2(4):303--314,
  1989.

\bibitem{Dick_et_al}
J.~Dick, F.~Kuo, and I.~Sloan.
\newblock High dimensional integration: The quasi-monte carlo way.
\newblock {\em Acta Numerica}, pages 133 -- 288, 2013.

\bibitem{MR3736669}
W.~E, J.~Han, and A.~Jentzen.
\newblock Deep learning-based numerical methods for high-dimensional parabolic
  partial differential equations and backward stochastic differential
  equations.
\newblock {\em Commun. Math. Stat.}, 5(4):349--380, 2017.

\bibitem{hochreiter1997long}
S.~Hochreiter and J.~Schmidhuber.
\newblock Long short-term memory.
\newblock {\em Neural computation}, 9(8):1735--1780, 1997.

\bibitem{hornik1991approximation}
K.~Hornik.
\newblock Approximation capabilities of multilayer feedforward networks.
\newblock {\em Neural networks}, 4(2):251--257, 1991.

\bibitem{MR2269875}
J.-M. Lasry and P.-L. Lions.
\newblock Jeux \`a champ moyen. {I}. {L}e cas stationnaire.
\newblock {\em C. R. Math. Acad. Sci. Paris}, 343(9):619--625, 2006.

\bibitem{MR2295621}
J.-M. Lasry and P.-L. Lions.
\newblock Mean field games.
\newblock {\em Jpn. J. Math.}, 2(1):229--260, 2007.

\bibitem{MR3501391}
M.~Lauri{\`e}re and O.~Pironneau.
\newblock Dynamic programming for mean-field type control.
\newblock {\em J. Optim. Theory Appl.}, 169(3):902--924, 2016.

\bibitem{MR1102015}
M.~Ledoux and M.~Talagrand.
\newblock {\em Probability in {B}anach spaces}, volume~23 of {\em Ergebnisse
  der Mathematik und ihrer Grenzgebiete (3) [Results in Mathematics and Related
  Areas (3)]}.
\newblock Springer-Verlag, Berlin, 1991.
\newblock Isoperimetry and processes.

\bibitem{Maurer}
A.~Maurer.
\newblock A vector-contraction inequality for {R}ademacher complexities, 2016.

\bibitem{MhaskarMicchelli}
H.~Mhaskar and C.~Micchelli.
\newblock Degree of approximation by neural and translation networks with a
  single hidden layer.
\newblock {\em Advances in Applied Mathematics}, 16:151--183, 1995.

\bibitem{MishuraVeretennikov}
Y.~S. Mishura and A.~Y. Veretennikov.
\newblock Existence and uniqueness theorems for solutions of
  {M}c{K}eanÐ{V}lasov stochastic equations.
\newblock Technical report, 2018.

\bibitem{MR3619691}
L.~Pfeiffer.
\newblock Numerical methods for mean-field type optimal control problems.
\newblock {\em Pure Appl. Funct. Anal.}, 1(4):629--655, 2016.

\bibitem{Ruthotto_et_al}
L.~Ruthotto, S.~Osher, W.~Li, L.~Nurbekyan, and S.~Fung.
\newblock A machine learning framework for solving high-dimensional mean field
  game and mean field control problems.
\newblock Technical report, 2019.

\bibitem{MR0251410}
M.~H. Schultz.
\newblock {$L^{\infty }$}-multivariate approximation theory.
\newblock {\em SIAM J. Numer. Anal.}, 6:184--209, 1969.

\bibitem{MR3874585}
J.~Sirignano and K.~Spiliopoulos.
\newblock D{GM}: a deep learning algorithm for solving partial differential
  equations.
\newblock {\em J. Comput. Phys.}, 375:1339--1364, 2018.

\bibitem{srivastava2015training}
R.~K. Srivastava, K.~Greff, and J.~Schmidhuber.
\newblock Training very deep networks.
\newblock In {\em Advances in neural information processing systems}, pages
  2377--2385, 2015.

\bibitem{Veretennikov}
A.~Y. Veretennikov.
\newblock On strong solutions and explicit formulas for solutions of stochastic
  integral equations.
\newblock {\em Mathematics of the USSR - Sbornik}, 39:387Ð403, 1981.

\bibitem{Veretennikov2}
A.~Y. Veretennikov.
\newblock On ergodic measures for {M}c{K}eanÐ{V}lasov stochastic equations.
\newblock In {\em Monte Carlo and Quasi- Monte Carlo Methods 2004}, pages
  471--486. Springer Verlag, 2006.

\bibitem{Yarykin}
P.~Yarykin.
\newblock Stability of the nonlinear stochastic process that approximates the
  system of interacting {B}rownian.
\newblock {\em Theory of Probability and Applications, Society for Industrial
  and Applied Mathematics}, 51(2):387 -- 396, 2007.

\end{thebibliography}

\end{document}